\documentclass[12pt,oneside,english,reqno]{amsart}
\usepackage[T1]{fontenc}
\usepackage[latin9]{inputenc}
\usepackage{geometry}
\usepackage{refstyle}
\usepackage{units}
\usepackage{amstext}
\usepackage{amsthm}
\usepackage{amssymb}
\usepackage{setspace}
\makeatletter

\AtBeginDocument{\providecommand\propref[1]{\ref{prop:#1}}}
\AtBeginDocument{\providecommand\corref[1]{\ref{cor:#1}}}
\AtBeginDocument{\providecommand\eqref[1]{\ref{eq:#1}}}
\AtBeginDocument{\providecommand\claimref[1]{\ref{claim:#1}}}
\AtBeginDocument{\providecommand\figref[1]{\ref{fig:#1}}}
\AtBeginDocument{\providecommand\defref[1]{\ref{def:#1}}}
\AtBeginDocument{\providecommand\lemref[1]{\ref{lem:#1}}}
\AtBeginDocument{\providecommand\remref[1]{\ref{rem:#1}}}
\AtBeginDocument{\providecommand\factref[1]{\ref{fact:#1}}}
\AtBeginDocument{\providecommand\thmref[1]{\ref{thm:#1}}}
\RS@ifundefined{subsecref}
  {\newref{subsec}{name = \RSsectxt}}
  {}
\RS@ifundefined{thmref}
  {\def\RSthmtxt{theorem~}\newref{thm}{name = \RSthmtxt}}
  {}
\RS@ifundefined{lemref}
  {\def\RSlemtxt{lemma~}\newref{lem}{name = \RSlemtxt}}
  {}

\numberwithin{equation}{section}
\numberwithin{figure}{section}
\theoremstyle{plain}
\newtheorem{thm}{\protect\theoremname}[section]
  \theoremstyle{remark}
  \newtheorem{rem}[thm]{\protect\remarkname}
  \theoremstyle{plain}
  \newtheorem{cor}[thm]{\protect\corollaryname}
  \theoremstyle{definition}
  \newtheorem{defn}[thm]{\protect\definitionname}
  \theoremstyle{plain}
  \newtheorem{prop}[thm]{\protect\propositionname}
  \theoremstyle{remark}
  \newtheorem{claim}[thm]{\protect\claimname}
  \theoremstyle{plain}
  \newtheorem{lem}[thm]{\protect\lemmaname}
  \theoremstyle{plain}
  \newtheorem{fact}[thm]{\protect\factname}
  \theoremstyle{plain}
  \newtheorem{conjecture}[thm]{\protect\conjecturename}

\date{}
\usepackage{amsmath}
\usepackage{bbm}
\usepackage{dsfont}
\usepackage{amssymb}

\newref{fact}{refcmd={Fact \ref{#1}}}
\newref{claim}{refcmd={Claim \ref{#1}}}
\newref{lem}{refcmd={Lemma \ref{#1}}}
\newref{thm}{refcmd={Theorem \ref{#1}}}
\newref{cor}{refcmd={Corollary \ref{#1}}}
\newref{prop}{refcmd={Proposition \ref{#1}}}
\newref{rem}{refcmd={Remark \ref{#1}}}
\newref{fig}{refcmd={Figure \ref{#1}}}
\newref{def}{refcmd={Definition \ref{#1}}}

\usepackage{tikz}

\usetikzlibrary{shapes.misc}
\usetikzlibrary{shapes.geometric}
\usetikzlibrary{arrows.meta}

\tikzset{cross/.style={cross out, draw=black, minimum size=2*(#1-\pgflinewidth), inner sep=0pt, outer sep=0pt}, cross/.default={1pt}}

\makeatother

\usepackage{babel}
  \providecommand{\claimname}{Claim}
  \providecommand{\conjecturename}{Conjecture}
  \providecommand{\corollaryname}{Corollary}
  \providecommand{\definitionname}{Definition}
  \providecommand{\factname}{Fact}
  \providecommand{\lemmaname}{Lemma}
  \providecommand{\propositionname}{Proposition}
  \providecommand{\remarkname}{Remark}
\providecommand{\theoremname}{Theorem}

\begin{document}
\global\long\def\One{\mathds{1}}

\global\long\def\Laplacian{\Delta}

\global\long\def\grad{\nabla}

\global\long\def\norm#1{\left\Vert #1\right\Vert }

\global\long\def\zz{\mathbb{Z}}

\global\long\def\rr{\mathbb{R}}

\global\long\def\nn{\mathbb{N}}

\global\long\def\pp{\mathbb{P}}

\global\long\def\ee{\mathbb{E}}

\global\long\def\floor#1{\left\lfloor #1\right\rfloor }

\title{Kinetically Constrained Models with Random Constraints}

\author{Assaf Shapira}
\begin{abstract}
We study two kinetically constrained models in a quenched random environment.
The first model is a mixed threshold Fredrickson-Andersen model on
$\zz^{2}$, where the update threshold is either $1$ or $2$. The
second is a mixture of the Fredrickson-Andersen $1$-spin facilitated
constraint and the North-East constraint in $\zz^{2}$. We compare
three time scales related to these models \textendash{} the bootstrap
percolation time for emptying the origin, the relaxation time of the
kinetically constrained model, and the time for emptying the origin
of the kinetically constrained model \textendash{} and understand
the effect of the random environment on each of them.
\end{abstract}

\maketitle

\section{Introduction}

Kinetically constrained models (KCMs) are a family of interacting
particle systems introduced by physicists in order to study glassy
and granular materials \cite{garrahansollichtoninelli2011kcm,Ritort}.
These are reversible Markov processes on the state space $\left\{ 0,1\right\} ^{V}$,
where $V$ is the set of vertices of some graph. The equilibrium measure
of these processes is a product measure of i.i.d. Bernoulli random
variables, and their nontrivial behavior is due to kinetic constraints
\textendash{} the state of each site is resampled at rate $1$, but
only when a certain local constraint is satisfied. This condition
expresses the fact that sites are blocked when there are not enough
empty sites in their vicinity. One example of such a constraint is
that of the Fredrickson-Andersen $j$-spin facilitated model on $\zz^{d}$,
in which an update is only possible if at least $j$ nearest neighbors
are empty \cite{FA1984kcm}. We will refer to this constraint as the
FA$j$f constraint. Another example is the North-East constraint:
the underlying graph is $\zz^{2}$, and an update is possible only
if both the site above and the site to the right are empty \cite{Lalley2006northeast}.
These constraints result in the lengthening of the time scales describing
the dynamics as the density of empty sites $q$ tends to $0$. This
happens since sites belonging to large occupied regions could only
be changed when empty sites penetrate from the outside. The main difficulty
in the analysis of KCMs is that they are not attractive, which prevents
us from using tools such as monotone coupling and censoring often
used in the study of Glauber dynamics. For this reason spectral analysis
and inequalities related to the spectral gap are essential for the
study of time scales in these models. See \cite{martinelli2018universality2d}
for more details.

A closely related family of models are the bootstrap percolation models,
which are, unlike KCMs, monotone deterministic processes in discrete
time. The state space of the bootstrap percolation is the same as
that of the KCM, and they share the same family of constraints; but
in the bootstrap percolation occupied sites become empty (deterministically)
whenever the constraint is satisfied, and empty sites can never be
filled. The initial conditions of the bootstrap percolation are random
i.i.d Bernoulli random variables with parameter $1-q$, i.e., they
are chosen according to the equilibrium measure of the KCM. In this
paper we will refer to the bootstrap percolation that corresponds
to certain constraint by their KCM name, so, for example, the $j$-neighbor
bootstrap percolation will be referred to as the bootstrap percolation
with the FA$j$f constraint.

In the examples given previously of the FA$j$f model and the North-East
model the constraints are translation invariant. Universality results
for general homogeneous models have been studied recently for the
bootstrap percolation in a series of works that provide a good understanding
of their behavior \cite{balisterbollobas2016subcriticalbp,bollobas2016universality2d,bollobassmithuzzell2015universalityzd,hartarsky2018mathcal}.
Inspired by the tools developed for the bootstrap percolation, universality
results on the KCMs could also be obtained for systems with general
translation invariant constraints \cite{martinellitonitelli2016towardsuniversality,martinelli2018universality2d}.
Another type of models vastly studied for the bootstrap percolation
are models in a random environment, such as the bootstrap percolation
on the polluted lattice \cite{gravner1997polluted,gravner2017polluted},
Galton-Watson trees \cite{bollobas2014bootstrapongwcriticality},
random regular graphs \cite{balogh2007randomregular,janson2009percolationexplosion}
and the Erd\H{o}s-R\'enyi graph \cite{janson2012Gnp}. KCMs in random
environments have also been studied in the physics literature, see
\cite{schulz1994modifiedfa,willart1999dynlengthscale}.

In this paper we will consider two models on the two dimensional lattice
with random constraints. We will focus on the divergence of time scales
when the equilibrium density of empty sites $q$ is small.

The time scale that is commonly considered in KCMs is the relaxation
time, i.e., the inverse of the spectral gap. This time scale determines
the slowest possible relaxation of correlation between observables,
but for homogeneous system it often coincides with typical time scales
of the system (see, e.g., \cite{mareche2018duarte}). However, when
the system is not homogeneous it will in general not describe actual
time scales that we observe. We will see in this paper that very unlikely
configurations of the disorder that appear far away determine the
relaxation time, even though the observed local behavior is not likely
to be effected by these remote regions.

Another time scale that is natural to look at is the first time at
which the origin (or any arbitrary vertex) is empty. In the bootstrap
percolation literature, this is indeed the time most commonly studied.
This time could be observed physically, and we will see that it is
not significantly effected by the ``bad'' regions far away from
the origin.

In this paper we compare the three time scales \textendash{} the time
it takes for the origin to be emptied with the bootstrap percolation,
the relaxation time for the KCM, and the first time the origin is
empty in the KCM. We will first analyze these time scales in a mixed
threshold FA model on the two dimensional lattice. The second model
we consider is a mixed KCM, in which vertices have either the FA1f
constraint, or the North-East constraint. This will be an example
of a model in which the relaxation time is infinite, but the time
at which the origin becomes empty is almost surely finite.

\section{Mixed threshold Fredrickson-Andersen model}

\subsection{Model and notation}

In this section we will treat two models \textendash{} the mixed threshold
bootstrap percolation on $\zz^{2}$ and the mixed threshold Fredrickson-Andersen
model on $\zz^{2}$. Both models will live on the same random environment,
that will determine the threshold at each vertex of $\zz^{2}$. It
will be denoted $\omega$, and the threshold at a vertex $x$ will
be $\omega_{x}\in\left\{ 1,2\right\} $. This environment is chosen
according to a measure $\nu$, which depends on a parameter $\pi\in\left(0,1\right)$.
$\nu$ will be a product measure \textendash{} for each vertex $x\in\zz^{2}$
$\omega_{x}$ equals $1$ with probability $\pi$ and $2$ with probability
$1-\pi$, independently from the other vertices. Sites with threshold
$1$ will be called \textit{easy}, and sites with threshold $2$ \textit{difficult}.

Both the bootstrap percolation and the FA dynamics are defined on
the state space $\Omega=\left\{ 0,1\right\} ^{\zz^{2}}$. For a configuration
$\eta\in\Omega$, we say that a site $x$ is \textit{empty} if $\eta_{x}=0$
and \textit{occupied} if $\eta_{x}=1$. For $\eta\in\Omega$ and $x\in\zz^{2}$
define the constraint
\begin{align}
c_{x}\left(\eta\right) & =\begin{cases}
1 & \sum_{y\sim x}\left(1-\eta_{y}\right)\ge\omega_{x}\\
0 & \text{otherwise}
\end{cases}.\label{eq:faconstraint}
\end{align}

We can now define the bootstrap percolation with these constraints
\textendash{} it is a deterministic dynamics in discrete time, where
at each time step $t$ empty vertices stay empty, and an occupied
vertex $x$ becomes empty if the constraint is satisfied, namely $c_{x}\left(\eta\left(t-1\right)\right)=1$.
The initial conditions for the bootstrap percolation are random, depending
on a parameter $q\in\left(0,1\right)$. They are chosen according
to the measure $\mu$, defined as a product of independent Bernoulli
measures:
\begin{align*}
\mu & =\bigotimes_{x\in\zz^{2}}\mu_{x},\\
\mu_{x} & \sim\text{Ber}\left(1-q\right).
\end{align*}

The Fredrickson-Andersen model is a continuous time Markov process
on $\Omega$. It is reversible with respect to the equilibrium measure
$\mu$ defined above, and its generator $\mathcal{L}$ is defined
by
\begin{equation}
\mathcal{L}f=\sum_{x}c_{x}\left(\mu_{x}f-f\right)\label{eq:generatorofkcm}
\end{equation}
for any local function $f$. We will denote by $\mathcal{D}$ the
Dirichlet form associated with $\mathcal{L}$. Probabilities and expected
values with respect to this process starting at $\eta$ will be denoted
by $\pp_{\eta}$ and $\ee_{\eta}$. When starting from equilibrium
we will use $\pp_{\mu}$ and $\ee_{\mu}$.

Finally, for any event $A\subseteq\Omega$, we define the hitting
time
\[
\tau_{A}=\inf\left\{ t\,:\,\eta(t)\in A\right\} .
\]
The hitting time is defined for both the KCM and the bootstrap percolation.
For the time it takes to empty the origin we will use the notation
\[
\tau_{0}=\tau_{\left\{ \eta_{0}=0\right\} }.
\]

\subsection{Results}

The first result concerns the bootstrap percolation. It will say that
for small values of $q$, $\tau_{0}$ scales as $\frac{1}{\sqrt{q}}$.
To avoid confusion we stress that $\mu$ and $\pp_{\mu}$ depend on
$q$, even though this dependence is not expressed explicitly in the
notation.
\begin{thm}
Consider the bootstrap percolation with the mixed FA constraint. Then
$\nu$-almost surely
\begin{align}
\lim_{q\rightarrow0}\,\mu\left[\tau_{0}\ge\frac{a}{\sqrt{q}}\right]\xrightarrow{a\rightarrow\infty}0,\label{eq:fa12d2bpupperbound}\\
\lim_{q\rightarrow0}\,\mu\left[\tau_{0}\le\frac{a}{\sqrt{q}}\right]\xrightarrow{a\rightarrow0}0.\label{eq:fa12d2bplowerbound}
\end{align}
\end{thm}

For the KCM we have an exponential divergence of the relaxation time,
but a power law behavior of $\tau_{0}$.
\begin{thm}
\label{thm:scalingoftimeforfa12}Consider the KCM with the mixed FA
constraint.
\begin{enumerate}
\item There exist a constant $c>0$ (that does not depend on $\pi,q$) such
that $\nu$-almost surely the relaxation time of the dynamics is at
least $e^{\nicefrac{c}{q}}$.
\item $\nu$-almost surely there exist $\underline{\alpha}$ and $\overline{\alpha}$
(which may depend on $\omega$) such that
\begin{align}
\pp_{\mu}\left[\tau_{0}\ge q^{-\overline{\alpha}}\right] & \xrightarrow{q\rightarrow0}0,\label{eq:fa12d2kcmupperbound}\\
\pp_{\mu}\left[\tau_{0}\le q^{-\underline{\alpha}}\right] & \xrightarrow{q\rightarrow0}0.\label{eq:fa12d2kcmlowerbound}
\end{align}

Moreover, $\ee_{\mu}\left[\tau_{0}\right]\ge q^{-\underline{\alpha}}$
for $q$ small enough.

\end{enumerate}
\end{thm}

\begin{rem}
\label{rem:exponentsaretruelyrandom}We will see that the two exponents
$\underline{\alpha}$ and $\overline{\alpha}$ cannot be deterministic
\textendash{} there is $\alpha_{0}\in\rr$ such that $\nu\left(\overline{\alpha}<\alpha_{0}\right)>0$
but $\nu\left(\underline{\alpha}<\alpha_{0}\right)<1$.
\end{rem}

\begin{rem}
In these two theorems we see that while $\tau_{0}$ for the bootstrap
percolation behaves like $q^{-\nicefrac{1}{2}}$, its scaling for
the FA is random. In the proof we will see in details the reason for
this difference, but we could already try to describe it heuristically.
The bootstrap percolation is dominated by the sites far away from
the origin, and once these sites are emptied the origin will be emptied
as well. The influence of the environment far away becomes deterministic
by a law of large numbers, so we do not see the randomness of $\omega$
in the exponent. To the contrary, in the FA dynamics even when sites
far away are empty, one must empty many sites in a close neighborhood
of the origin simultaneously before the origin could be emptied. Therefore,
in order to empty the origin we must overcome a large energy barrier,
which makes $\tau_{0}$ bigger. This effect depends on the structure
close to the origin, so it feels the randomness of the environment.
\end{rem}

For simplicity, we have chosen to focus on the two dimensional case.
However, a more general result can also be obtained. In the next two
theorems we will consider the bootstrap percolation and KCM on $\zz^{d}$.
The thresholds $\left\{ \omega_{x}\right\} _{i\in\zz^{2}}$ are i.i.d.,
according to a law that we denote by $\nu$. We will also assume that
the probability that the threshold is $1$ is nonzero, and that the
probability that the threshold is more than $d$ is zero.
\begin{thm}
For the bootstrap percolation model described above, $\nu$-almost
surely
\begin{align}
\lim_{q\rightarrow0}\,\mu\left[\tau_{0}\ge aq^{-\nicefrac{1}{d}}\right]\xrightarrow{a\rightarrow\infty}0,\label{eq:fa1jbpupperbound}\\
\lim_{q\rightarrow0}\,\mu\left[\tau_{0}\le aq^{-\nicefrac{1}{d}}\right]\xrightarrow{a\rightarrow0}0.\label{eq:fa1jbplowerbound}
\end{align}
\end{thm}

\begin{thm}
For the KCM described above, $\nu$-almost surely there exist $\underline{\alpha}$
and $\overline{\alpha}$ (which may depend on $\omega$) such that
\begin{align}
\pp_{\mu}\left[\tau_{0}\ge q^{-\overline{\alpha}}\right] & \xrightarrow{q\rightarrow0}0,\label{eq:fa1jkcmupperbound}\\
\pp_{\mu}\left[\tau_{0}\le q^{-\underline{\alpha}}\right] & \xrightarrow{q\rightarrow0}0.\label{eq:fa1jkcmlowerbound}
\end{align}
\end{thm}

\section{Mixed north-east and FA1$f$ KCM}

\subsection{Model and notation}

In this section we will consider again a kinetically constrained dynamics
in an environment with mixed constraints. This time, however, the
two constraints we will have are FA1f and north-east. That is, using
the same $\omega$ and $\nu$ as before, for $x$ such that $\omega_{x}=1$

\begin{align*}
c_{x}\left(\eta\right) & =\begin{cases}
1 & \sum_{y\sim x}\left(1-\eta_{y}\right)\ge1\\
0 & \text{otherwise}
\end{cases},
\end{align*}
and when $\omega_{x}=2$

\begin{align*}
c_{x}\left(\eta\right) & =\begin{cases}
1 & \eta_{x+e_{1}}=0\text{ and }\eta_{x+e_{2}}=0\\
0 & \text{otherwise}
\end{cases}.
\end{align*}
For the same $\mu$, we can define $\mathcal{L}$ as in \ref{eq:generatorofkcm}.
Note that $c_{x}$ (and therefore $\mathcal{L}$) are not the same
as those of the previous section, even though we use the same letters
to describe them. Again, the hitting time of a set $A$ will be denoted
by $\tau_{A}$, and $\tau_{0}=\tau_{\left\{ \eta_{0}=0\right\} }$.

We restrict ourselves to the case where $\pi$ is greater than the
critical probability for the Bernoulli site percolation on $\zz^{2}$,
denoted by $p^{\text{SP }}$. The critical probability for the oriented
percolation on $\zz^{2}$ will be denoted by $p^{\text{OP}}$.
\begin{rem}
Our choice of regime, where easy sites percolate, guarantees that
all sites are emptiable for the bootstrap percolation. The infinite
cluster $\mathcal{C}$ of easy sites is emptiable since it must contain
an empty site somewhere. The connected components of $\zz^{2}\setminus\mathcal{C}$
are finite, and have an emptiable boundary, so each of them will also
be emptied eventually.

This choice, however, is not the only one for which all sites are
emptiable. For any fixed environment $\omega$ there is a critical
value $q_{c}$ such that above $q_{c}$ all sites are emptiable and
below $q_{c}$ some sites remain occupied forever. For $\pi>p^{\text{SP}}$
we already know that $\nu$-almost surely $q_{c}=0$. In fact, the
same argument gives a slightly better result by allowing sites to
be difficult if they are also empty. This implies that $q_{c}\le1-\frac{1-p^{\text{SP}}}{1-\pi}$.
On the other hand, if there is an infinite up-right path of difficult
sites that are all occupied, this path could never be emptied. This
will imply that $q_{c}\ge1-\frac{p^{\text{OP}}}{1-\pi}$.
\end{rem}

\subsection{Results}

We will see for this model that it is possible to have an infinite
relaxation time, and still the tail of the distribution of $\tau_{0}$
decays exponentially, with a rate that scales polynomially with $q$.
\begin{thm}
Consider the kinetically constrained model described above, with $\pi>p^{\text{SP}}$
and $q\le q^{\text{OP}}$.
\begin{enumerate}
\item $\nu$-almost surely the spectral gap is $0$, i.e., the relaxation
time is infinite.
\item There exist two positive constants $c,C$ depending on $\pi$ and
a $\nu$-random variable $\tau$ such that
\begin{enumerate}
\item $\pp_{\mu}\left(\tau_{0}\ge t\right)\le e^{-\nicefrac{t}{\tau}}$
for all $t>0$,
\item $\nu\left(\tau\ge t\right)\le C\,t^{\frac{c}{\log q}}$ for $t$ large
enough.
\end{enumerate}
\end{enumerate}
\end{thm}

\section{Some tools}

In this section we will present some tools that will help us analyze
the kinetically constrained models that we have introduced. We will
start by considering a general state space $\Omega$, and any Markov
process on $\Omega$ that is reversible with respect to a certain
measure $\mu$. We denote its generator by $\mathcal{L}$ and the
associated Dirichlet form by $\mathcal{D}$. We will consider, for
some event $A$, its hitting time $\tau_{A}$. With some abuse of
notation, we use $\tau_{A}$ also for the $\mu$-random variable giving
for every state $\eta\in\Omega$ the expected hitting time at $A$
starting from that state:
\[
\tau_{A}\left(\eta\right)=\ee_{\eta}\left(\tau_{A}\right).
\]

$\tau_{A}\left(\eta\right)$ satisfies the following Poisson problem:
\begin{align}
\mathcal{L}\tau_{A} & =-1\text{ on }A^{c},\label{eq:poissonproblem}\\
\tau_{A} & =0\text{ on }A.\nonumber 
\end{align}

By multiplying both sides of the equation by $\tau_{A}$ and integrating
with respect to $\mu$, we obtain
\begin{cor}
\label{cor:dirichletequalsexpectation}$\mu\left(\tau_{A}\right)=\mathcal{D}\tau_{A}$.
\end{cor}

Rewriting this corollary as $\mu\left(\tau_{A}\right)=\frac{\mu\left(\tau_{A}\right)^{2}}{\mathcal{D}\tau_{A}}$,
it resembles a variational principle introduced in \cite{AsselhaDaiPra2001quasistationary}
that will be useful in the following. In order to formulate it we
will need to introduce some notation.
\begin{defn}
For an event $A\subseteq\Omega$, $V_{A}$ is the set of all functions
in the domain of $\mathcal{L}$ that vanish on the event $A$. Note
that, in particular, $\tau_{A}\in V_{A}$.
\end{defn}

\begin{defn}
\label{def:taubar}For an event $A\subseteq\Omega$,
\[
\overline{\tau}_{A}=\sup_{0\neq f\in V_{A}}\,\frac{\mu\left(f^{2}\right)}{\mathcal{D}f}.
\]
\end{defn}

The following proposition is given in the first equation of the proof
of Theorem 2 in \cite{AsselhaDaiPra2001quasistationary}:
\begin{prop}
\label{prop:exponentialdecayofhittingtime} $\pp_{\mu}\left[\tau_{A}>t\right]\le e^{-t/\overline{\tau}_{A}}$.
\end{prop}

\begin{rem}
In particular, \propref{exponentialdecayofhittingtime} implies that
$\mu\left(\tau_{A}\right)\le\bar{\tau}_{A}$. This, however, could
be derived much more simply from \corref{dirichletequalsexpectation}
\textendash 
\[
\mu\left(\tau_{A}\right)^{2}\le\mu\left(\tau_{A}^{2}\right)\le\overline{\tau}_{A}\mathcal{D}\tau_{A}=\overline{\tau}_{A}\mu\left(\tau_{A}\right).
\]
Note that whenever $\tau_{A}$ is not constant on $A^{c}$ this inequality
is strict. Thus on one hand \propref{exponentialdecayofhittingtime}
gives an exponential decay of $\pp_{\mu}\left[\tau_{A}>t\right]$,
which is stronger than the information on the expected value we can
obtain from the Poisson problem in \eqref{poissonproblem}. On the
other hand, $\overline{\tau}_{A}$ could be longer than the actual
expectation of $\tau_{A}$.
\end{rem}

In order to bound the hitting time from below we will formulate a
variational principle that will characterize $\tau_{A}$.
\begin{defn}
For $f\in V_{A}$, let
\[
\mathcal{T}f=2\mu\left(f\right)-\mathcal{D}f.
\]
\end{defn}

\begin{prop}
\label{prop:variationalprincipleforhittingtime}$\tau_{A}$ maximizes
$\mathcal{T}$ in $V_{A}$. Moreover, $\mu\left(\tau_{A}\right)=\sup_{f\in V_{A}}\mathcal{T}f$.
\end{prop}

\begin{proof}
Consider $f\in V_{A}$, and let $\delta=f-\tau_{A}$. Using the self-adjointness
of $\mathcal{L}$, \eqref{poissonproblem}, and the fact that $\delta\in V_{A}$
we obtain
\begin{align*}
\mathcal{T}f & =\mathcal{T}\left(\tau_{A}+\delta\right)\\
 & =2\mu\left(\tau_{A}\right)+2\mu\left(\delta\right)-\mathcal{D}\tau_{A}-\mathcal{D}\delta+2\mu\left(\delta\mathcal{L}\tau\right)\\
 & =\mathcal{T}\tau_{A}-\mathcal{D}\delta.
\end{align*}
By the positivity of the Dirichlet form, $\mathcal{T}$ is indeed
maximized by $\tau_{A}$. Finally, by \corref{dirichletequalsexpectation},
\begin{align*}
\sup_{f\in V_{A}}\mathcal{T}f & =\mathcal{T}\tau_{A}=2\mu\left(\tau_{A}\right)-\mathcal{D}\tau_{A}=\mu\left(\tau_{A}\right).
\end{align*}
\end{proof}
As an immediate consequence we can deduce the monotonicity of the
expected hitting time:
\begin{cor}
\label{cor:monotonicityofhittingtime}Let $\mathcal{D}$ and $\mathcal{D}^{\prime}$
be the Dirichlet forms of two reversible Markov processes defined
on the same space, such that both share the same equilibrium measure
$\mu$. We denote the expectations with respect to these processes
starting at equilibrium by $\ee_{\mu}$ and $\ee_{\mu}^{\prime}$.
Assume that the domain of $\mathcal{D}$ is contained in the domain
of $\mathcal{D}^{\prime}$, and that for every $f\in\text{Dom}\mathcal{D}$
\[
\mathcal{D}f\le\mathcal{D}^{\prime}f.
\]
Then, for an event $A\subseteq\Omega$,
\[
\ee_{\mu}\tau_{A}\le\ee_{\mu}^{\prime}\tau_{A}.
\]
\end{cor}

We will now restrict ourselves to kinetically constrained models.
Fix a graph $G$ and take $\Omega=\left\{ 0,1\right\} ^{G}$. For
every vertex $x\in G$ and a state $\eta\in\Omega$ we define a constraint
$c_{x}\left(\eta\right)\in\left\{ 0,1\right\} $. The constraint does
not depend on the value at $x$, and is non-increasing in $\eta$.
The equilibrium measure $\mu$ is a product measure. The generator
of this process, operating on a local function $f$, is given by
\[
\mathcal{L}f=\sum_{x}c_{x}\left(\mu_{x}f-f\right)
\]
and its Dirichlet form by
\[
\mathcal{D}f=\mu\left(\sum_{x}c_{x}\text{Var}_{x}f\right).
\]

Fix a subgraph $H$ of $G$, and denote the complement of $H$ in
$G$ by $H^{c}$.

We will compare the dynamics of this KCM to the dynamics restricted
to $H$, with boundary conditions that are the most constrained ones.
\begin{defn}
\label{def:restricteddynamics}The restricted dynamics on $H$ is
the KCM defined by the constraints
\[
c_{x}^{H}\left(\eta\right)=c_{x}\left(\eta^{H}\right),
\]
where, for $\eta\in\left\{ 0,1\right\} ^{H}$, $\eta^{H}$ is the
configuration given by

\[
\eta^{H}(x)=\begin{cases}
\eta_{x} & x\in H\\
1 & x\in H^{c}
\end{cases}.
\]
We will denote the corresponding generator by $\mathcal{L}_{H}$ and
its Dirichlet form by $\mathcal{D}_{H}$.
\end{defn}

\begin{claim}
\label{claim:restrictingdirichlet}For any $f$ in the domain of $\mathcal{L}$
\[
\mathcal{D}f\ge\mu_{H^{c}}\mathcal{D}_{H}f.
\]
\end{claim}

\begin{proof}
$c_{x}^{H}\le c_{x}$ and $\text{Var}_{x}f$ is positive, therefore
\begin{align*}
\mathcal{D}f & =\mu\left(\sum_{x}c_{x}\text{Var}_{x}f\right)\ge\mu\left(\sum_{x\in H}c_{x}^{H}\text{Var}_{x}f\right).
\end{align*}
\end{proof}
The next claim will allow us to relate the spectral gap of the restricted
dynamics to the variational principles discussed earlier.
\begin{claim}
\label{claim:gapofrestricteddynamics}Let $\gamma_{H}$ be the spectral
gap of $\mathcal{L}_{H}$, and fix an event $A$ that depends only
on the occupation of the vertices of $H$. Then for all $f\in V_{A}$
\end{claim}

\begin{enumerate}
\item $\mathcal{D}f\ge\mu\left(A\right)\gamma_{H}\,\left(\mu f\right)^{2}$,
\item $\mathcal{D}f\ge\frac{\mu\left(A\right)}{1+\mu\left(A\right)}\gamma_{H}\,\mu\left(f^{2}\right)$
\end{enumerate}
\begin{proof}
First, note that $\mu_{H}\left(A\right)\le\mu_{H}\left(f=0\right)\le\mu_{H}\left(\left|f-\mu_{H}f\right|\ge\mu_{H}f\right)$.
Therefore, by Chebyshev inequality and the fact that $\mu\left(A\right)=\mu_{H}\left(A\right)$,
\begin{equation}
\mu\left(A\right)\le\frac{\text{Var}_{H}f}{\left(\mu_{H}f\right)^{2}}.\label{eq:boundingvarianceoverexpectaitionsquared}
\end{equation}
Then, \claimref{restrictingdirichlet} implies
\begin{align*}
\mathcal{D}f & \ge\mu_{H^{c}}\mathcal{D}_{H}f\ge\gamma_{H}\mu_{H^{c}}\text{Var}_{H}f\ge\mu\left(A\right)\gamma_{H}\,\mu_{H^{c}}\left(\mu_{H}f\right)^{2}\ge\mu\left(A\right)\gamma_{H}\,\left(\mu f\right)^{2}
\end{align*}
by Jensen inequality. For the second part, we use inequality \ref{eq:boundingvarianceoverexpectaitionsquared}
\[
\text{Var}_{H}f\ge\mu\left(A\right)\left(\mu_{H}\left(f^{2}\right)-\text{Var}_{H}f\right),
\]
which implies 
\[
\text{Var}_{H}f\ge\frac{\mu\left(A\right)}{1+\mu\left(A\right)}\mu_{H}\left(f^{2}\right).
\]
The result then follows by applying \claimref{restrictingdirichlet}.
\end{proof}

\section{Proof of the results}

\subsection{Mixed threshold bootstrap percolation on $\protect\zz^{2}$}

\subsubsection{Proof of \eqref{fa12d2bpupperbound}}

For the upper bound we will find a specific mechanism in which a cluster
of empty sites could grow until it reaches the origin.
\begin{defn}
\label{def:goodsquareforfa12}A square (that is, a subset of $\zz^{2}$
of the form $x+\left[L\right]^{2}$) is \textit{good} if it contains
at least one easy site in each line and in each column.
\end{defn}

\begin{claim}
\label{claim:squaresaregood}Fix $L$. The probability that a square
of side $L$ is good is at least $1-2Le^{-\pi L}$.
\end{claim}

\begin{proof}
\begin{align*}
\pp\left[\text{easy site in each line}\right] & =\left[1-\left(1-\pi\right)^{L}\right]^{L}\ge1-Le^{-\pi L}.
\end{align*}
The same bound holds for $\pp\left[\text{easy site in each column}\right]$,
and then we conclude by the union bound.
\end{proof}
\begin{defn}
The square $\left[L\right]^{2}$ is \textit{excellent} if for every
$2\le i\le L$ at least one of the sites in $\left\{ i\right\} \times\left[i-1\right]$
is easy, and at least one of the sites in $\left[i-1\right]\times\left\{ i\right\} $
is easy. For other squares of side $L$ being excellent is defined
by translation.
\end{defn}

We will use $p_{L}$ to denote the probability that a square of side
$L$ is excellent. Note that $p_{L}$ depends only on $\pi$ and not
on $q$.

\begin{figure}
\begin{tikzpicture}[scale=0.3, every node/.style={scale=0.5}]
	\def\xx{0}
	\def\yy{0}
	
	\draw[step=5, black, very thin,xshift=\xx cm, yshift=\yy cm] (0,0) grid +(5,5);
	\draw[step=5, black, very thin,xshift=\xx cm, yshift=\yy cm] (0,5) grid +(5,5);
	\draw[step=5, black, very thin,xshift=\xx cm, yshift=\yy cm] (5,5) grid +(5,5);
	
	\draw[step=1, gray, very thin,xshift=\xx cm, yshift=\yy cm] (0,0) grid +(5,5);
	\draw[step=1, gray, very thin,xshift=\xx cm, yshift=\yy cm] (0,5) grid +(5,5);
	\draw[step=1, gray, very thin,xshift=\xx cm, yshift=\yy cm] (5,5) grid +(5,5);
	
	\draw (\xx+0.5,\yy+0.5) node[black] {$0$};
	
	\draw (\xx+0.5,\yy+1.5) node[black] {$e$};
	\draw (\xx+1.5,\yy+0.5) node[black] {$e$};
	\draw (\xx+2.5,\yy+1.5) node[black] {$e$};
	\draw (\xx+1.5,\yy+2.5) node[black] {$e$};
	\draw (\xx+2.5,\yy+3.5) node[black] {$e$};
	\draw (\xx+3.5,\yy+0.5) node[black] {$e$};
	\draw (\xx+1.5,\yy+4.5) node[black] {$e$};
	\draw (\xx+4.5,\yy+2.5) node[black] {$e$};
	
	\draw (\xx+0.5,\yy+9.5) node[black] {$e$};
	\draw (\xx+1.5,\yy+7.5) node[black] {$e$};
	\draw (\xx+2.5,\yy+8.5) node[black] {$e$};
	\draw (\xx+3.5,\yy+5.5) node[black] {$e$};
	\draw (\xx+4.5,\yy+6.5) node[black] {$e$};
	
	\draw (\xx+5.5,\yy+7.5) node[black] {$e$};
	\draw (\xx+6.5,\yy+5.5) node[black] {$e$};
	\draw (\xx+7.5,\yy+6.5) node[black] {$e$};
	\draw (\xx+8.5,\yy+8.5) node[black] {$e$};
	\draw (\xx+9.5,\yy+9.5) node[black] {$e$};

	\draw[->]  (11,5) to (13,5);

	\def\xx{14}
	\def\yy{0}
	
	\draw[step=5, black, very thin,xshift=\xx cm, yshift=\yy cm] (0,0) grid +(5,5);
	\draw[step=5, black, very thin,xshift=\xx cm, yshift=\yy cm] (0,5) grid +(5,5);
	\draw[step=5, black, very thin,xshift=\xx cm, yshift=\yy cm] (5,5) grid +(5,5);
	
	\draw[step=1, gray, very thin,xshift=\xx cm, yshift=\yy cm] (0,0) grid +(5,5);
	\draw[step=1, gray, very thin,xshift=\xx cm, yshift=\yy cm] (0,5) grid +(5,5);
	\draw[step=1, gray, very thin,xshift=\xx cm, yshift=\yy cm] (5,5) grid +(5,5);
	
	\draw (\xx+0.5,\yy+0.5) node[black] {$0$};
	
	\draw (\xx+0.5,\yy+1.5) node[black] {$0$};
	\draw (\xx+1.5,\yy+0.5) node[black] {$0$};
	\draw (\xx+2.5,\yy+1.5) node[black] {$e$};
	\draw (\xx+1.5,\yy+2.5) node[black] {$e$};
	\draw (\xx+2.5,\yy+3.5) node[black] {$e$};
	\draw (\xx+3.5,\yy+0.5) node[black] {$e$};
	\draw (\xx+1.5,\yy+4.5) node[black] {$e$};
	\draw (\xx+4.5,\yy+2.5) node[black] {$e$};
	
	\draw (\xx+0.5,\yy+9.5) node[black] {$e$};
	\draw (\xx+1.5,\yy+7.5) node[black] {$e$};
	\draw (\xx+2.5,\yy+8.5) node[black] {$e$};
	\draw (\xx+3.5,\yy+5.5) node[black] {$e$};
	\draw (\xx+4.5,\yy+6.5) node[black] {$e$};
	
	\draw (\xx+5.5,\yy+7.5) node[black] {$e$};
	\draw (\xx+6.5,\yy+5.5) node[black] {$e$};
	\draw (\xx+7.5,\yy+6.5) node[black] {$e$};
	\draw (\xx+8.5,\yy+8.5) node[black] {$e$};
	\draw (\xx+9.5,\yy+9.5) node[black] {$e$};

	\draw[->]  (25,5) to (27,5);

	\def\xx{28}
	\def\yy{0}

	\draw[step=5, black, very thin,xshift=\xx cm, yshift=\yy cm] (0,0) grid +(5,5);
	\draw[step=5, black, very thin,xshift=\xx cm, yshift=\yy cm] (0,5) grid +(5,5);
	\draw[step=5, black, very thin,xshift=\xx cm, yshift=\yy cm] (5,5) grid +(5,5);
	
	\draw[step=1, gray, very thin,xshift=\xx cm, yshift=\yy cm] (0,0) grid +(5,5);
	\draw[step=1, gray, very thin,xshift=\xx cm, yshift=\yy cm] (0,5) grid +(5,5);
	\draw[step=1, gray, very thin,xshift=\xx cm, yshift=\yy cm] (5,5) grid +(5,5);
	
	\draw (\xx+0.5,\yy+0.5) node[black] {$0$};
	
	\draw (\xx+0.5,\yy+1.5) node[black] {$0$};
	\draw (\xx+1.5,\yy+0.5) node[black] {$0$};
	\draw (\xx+1.5,\yy+1.5) node[black] {$0$};
	\draw (\xx+2.5,\yy+1.5) node[black] {$e$};
	\draw (\xx+1.5,\yy+2.5) node[black] {$e$};
	\draw (\xx+2.5,\yy+3.5) node[black] {$e$};
	\draw (\xx+3.5,\yy+0.5) node[black] {$e$};
	\draw (\xx+1.5,\yy+4.5) node[black] {$e$};
	\draw (\xx+4.5,\yy+2.5) node[black] {$e$};
	
	\draw (\xx+0.5,\yy+9.5) node[black] {$e$};
	\draw (\xx+1.5,\yy+7.5) node[black] {$e$};
	\draw (\xx+2.5,\yy+8.5) node[black] {$e$};
	\draw (\xx+3.5,\yy+5.5) node[black] {$e$};
	\draw (\xx+4.5,\yy+6.5) node[black] {$e$};
	
	\draw (\xx+5.5,\yy+7.5) node[black] {$e$};
	\draw (\xx+6.5,\yy+5.5) node[black] {$e$};
	\draw (\xx+7.5,\yy+6.5) node[black] {$e$};
	\draw (\xx+8.5,\yy+8.5) node[black] {$e$};
	\draw (\xx+9.5,\yy+9.5) node[black] {$e$};

	\draw[->]  (39,5) to (41,5);

	\def\xx{42}
	\def\yy{0}

	\draw[step=5, black, very thin,xshift=\xx cm, yshift=\yy cm] (0,0) grid +(5,5);
	\draw[step=5, black, very thin,xshift=\xx cm, yshift=\yy cm] (0,5) grid +(5,5);
	\draw[step=5, black, very thin,xshift=\xx cm, yshift=\yy cm] (5,5) grid +(5,5);
	
	\draw[step=1, gray, very thin,xshift=\xx cm, yshift=\yy cm] (0,0) grid +(5,5);
	\draw[step=1, gray, very thin,xshift=\xx cm, yshift=\yy cm] (0,5) grid +(5,5);
	\draw[step=1, gray, very thin,xshift=\xx cm, yshift=\yy cm] (5,5) grid +(5,5);
	
	\draw (\xx+0.5,\yy+0.5) node[black] {$0$};
	
	\draw (\xx+0.5,\yy+1.5) node[black] {$0$};
	\draw (\xx+1.5,\yy+0.5) node[black] {$0$};
	\draw (\xx+1.5,\yy+1.5) node[black] {$0$};
	\draw (\xx+2.5,\yy+1.5) node[black] {$0$};
	\draw (\xx+1.5,\yy+2.5) node[black] {$0$};
	\draw (\xx+2.5,\yy+3.5) node[black] {$e$};
	\draw (\xx+3.5,\yy+0.5) node[black] {$e$};
	\draw (\xx+1.5,\yy+4.5) node[black] {$e$};
	\draw (\xx+4.5,\yy+2.5) node[black] {$e$};
	
	\draw (\xx+0.5,\yy+9.5) node[black] {$e$};
	\draw (\xx+1.5,\yy+7.5) node[black] {$e$};
	\draw (\xx+2.5,\yy+8.5) node[black] {$e$};
	\draw (\xx+3.5,\yy+5.5) node[black] {$e$};
	\draw (\xx+4.5,\yy+6.5) node[black] {$e$};
	
	\draw (\xx+5.5,\yy+7.5) node[black] {$e$};
	\draw (\xx+6.5,\yy+5.5) node[black] {$e$};
	\draw (\xx+7.5,\yy+6.5) node[black] {$e$};
	\draw (\xx+8.5,\yy+8.5) node[black] {$e$};
	\draw (\xx+9.5,\yy+9.5) node[black] {$e$};

	\draw[->]  (53,5) to (55,5);

	\def\xx{0}
	\def\yy{-13}

	\draw[step=5, black, very thin,xshift=\xx cm, yshift=\yy cm] (0,0) grid +(5,5);
	\draw[step=5, black, very thin,xshift=\xx cm, yshift=\yy cm] (0,5) grid +(5,5);
	\draw[step=5, black, very thin,xshift=\xx cm, yshift=\yy cm] (5,5) grid +(5,5);
	
	\draw[step=1, gray, very thin,xshift=\xx cm, yshift=\yy cm] (0,0) grid +(5,5);
	\draw[step=1, gray, very thin,xshift=\xx cm, yshift=\yy cm] (0,5) grid +(5,5);
	\draw[step=1, gray, very thin,xshift=\xx cm, yshift=\yy cm] (5,5) grid +(5,5);

	\foreach \x in {0,...,2}{
		\foreach \y in {0,...,2}{
			\draw (\xx+\x+0.5,\yy+\y+0.5) node[black] {$0$};
		}
	}	

	\draw (\xx+2.5,\yy+3.5) node[black] {$e$};
	\draw (\xx+3.5,\yy+0.5) node[black] {$e$};
	\draw (\xx+1.5,\yy+4.5) node[black] {$e$};
	\draw (\xx+4.5,\yy+2.5) node[black] {$e$};
	
	\draw (\xx+0.5,\yy+9.5) node[black] {$e$};
	\draw (\xx+1.5,\yy+7.5) node[black] {$e$};
	\draw (\xx+2.5,\yy+8.5) node[black] {$e$};
	\draw (\xx+3.5,\yy+5.5) node[black] {$e$};
	\draw (\xx+4.5,\yy+6.5) node[black] {$e$};
	
	\draw (\xx+5.5,\yy+7.5) node[black] {$e$};
	\draw (\xx+6.5,\yy+5.5) node[black] {$e$};
	\draw (\xx+7.5,\yy+6.5) node[black] {$e$};
	\draw (\xx+8.5,\yy+8.5) node[black] {$e$};
	\draw (\xx+9.5,\yy+9.5) node[black] {$e$};

	\draw[->]  (11,-8) to (13,-8);

	\def\xx{14}
	\def\yy{-13}

	\draw[step=5, black, very thin,xshift=\xx cm, yshift=\yy cm] (0,0) grid +(5,5);
	\draw[step=5, black, very thin,xshift=\xx cm, yshift=\yy cm] (0,5) grid +(5,5);
	\draw[step=5, black, very thin,xshift=\xx cm, yshift=\yy cm] (5,5) grid +(5,5);
	
	\draw[step=1, gray, very thin,xshift=\xx cm, yshift=\yy cm] (0,0) grid +(5,5);
	\draw[step=1, gray, very thin,xshift=\xx cm, yshift=\yy cm] (0,5) grid +(5,5);
	\draw[step=1, gray, very thin,xshift=\xx cm, yshift=\yy cm] (5,5) grid +(5,5);

	\foreach \x in {0,...,3}{
		\foreach \y in {0,...,3}{
			\draw (\xx+\x+0.5,\yy+\y+0.5) node[black] {$0$};
		}
	}	

	\draw (\xx+1.5,\yy+4.5) node[black] {$e$};
	\draw (\xx+4.5,\yy+2.5) node[black] {$e$};
	
	\draw (\xx+0.5,\yy+9.5) node[black] {$e$};
	\draw (\xx+1.5,\yy+7.5) node[black] {$e$};
	\draw (\xx+2.5,\yy+8.5) node[black] {$e$};
	\draw (\xx+3.5,\yy+5.5) node[black] {$e$};
	\draw (\xx+4.5,\yy+6.5) node[black] {$e$};
	
	\draw (\xx+5.5,\yy+7.5) node[black] {$e$};
	\draw (\xx+6.5,\yy+5.5) node[black] {$e$};
	\draw (\xx+7.5,\yy+6.5) node[black] {$e$};
	\draw (\xx+8.5,\yy+8.5) node[black] {$e$};
	\draw (\xx+9.5,\yy+9.5) node[black] {$e$};

	\draw[->]  (25,-8) to (27,-8);

	\def\xx{28}
	\def\yy{-13}
	
	\draw[step=5, black, very thin,xshift=\xx cm, yshift=\yy cm] (0,0) grid +(5,5);
	\draw[step=5, black, very thin,xshift=\xx cm, yshift=\yy cm] (0,5) grid +(5,5);
	\draw[step=5, black, very thin,xshift=\xx cm, yshift=\yy cm] (5,5) grid +(5,5);
	
	\draw[step=1, gray, very thin,xshift=\xx cm, yshift=\yy cm] (0,0) grid +(5,5);
	\draw[step=1, gray, very thin,xshift=\xx cm, yshift=\yy cm] (0,5) grid +(5,5);
	\draw[step=1, gray, very thin,xshift=\xx cm, yshift=\yy cm] (5,5) grid +(5,5);
	
	\foreach \x in {0,...,4}{
		\foreach \y in {0,...,4}{
			\draw (\xx+\x+0.5,\yy+\y+0.5) node[black] {$0$};
		}
	}	
	
	\draw (\xx+0.5,\yy+9.5) node[black] {$e$};
	\draw (\xx+1.5,\yy+7.5) node[black] {$e$};
	\draw (\xx+2.5,\yy+8.5) node[black] {$e$};
	\draw (\xx+3.5,\yy+5.5) node[black] {$e$};
	\draw (\xx+4.5,\yy+6.5) node[black] {$e$};
	
	\draw (\xx+5.5,\yy+7.5) node[black] {$e$};
	\draw (\xx+6.5,\yy+5.5) node[black] {$e$};
	\draw (\xx+7.5,\yy+6.5) node[black] {$e$};
	\draw (\xx+8.5,\yy+8.5) node[black] {$e$};
	\draw (\xx+9.5,\yy+9.5) node[black] {$e$};
	
	\draw[->]  (39,-8) to (41,-8);

	\def\xx{42}
	\def\yy{-13}

	\draw[step=5, black, very thin,xshift=\xx cm, yshift=\yy cm] (0,0) grid +(5,5);
	\draw[step=5, black, very thin,xshift=\xx cm, yshift=\yy cm] (0,5) grid +(5,5);
	\draw[step=5, black, very thin,xshift=\xx cm, yshift=\yy cm] (5,5) grid +(5,5);
	
	\draw[step=1, gray, very thin,xshift=\xx cm, yshift=\yy cm] (0,0) grid +(5,5);
	\draw[step=1, gray, very thin,xshift=\xx cm, yshift=\yy cm] (0,5) grid +(5,5);
	\draw[step=1, gray, very thin,xshift=\xx cm, yshift=\yy cm] (5,5) grid +(5,5);

	\foreach \x in {0,...,4}{
		\foreach \y in {0,...,4}{
			\draw (\xx+\x+0.5,\yy+\y+0.5) node[black] {$0$};
		}
	}

	\draw (\xx+0.5,\yy+9.5) node[black] {$e$};
	\draw (\xx+1.5,\yy+7.5) node[black] {$e$};
	\draw (\xx+2.5,\yy+8.5) node[black] {$e$};
	\draw (\xx+3.5,\yy+5.5) node[black] {$0$};
	\draw (\xx+4.5,\yy+6.5) node[black] {$e$};
	
	\draw (\xx+5.5,\yy+7.5) node[black] {$e$};
	\draw (\xx+6.5,\yy+5.5) node[black] {$e$};
	\draw (\xx+7.5,\yy+6.5) node[black] {$e$};
	\draw (\xx+8.5,\yy+8.5) node[black] {$e$};
	\draw (\xx+9.5,\yy+9.5) node[black] {$e$};

	\draw[->]  (53,-8) to (55,-8);

	\def\xx{0}
	\def\yy{-26}

	\draw[step=5, black, very thin,xshift=\xx cm, yshift=\yy cm] (0,0) grid +(5,5);
	\draw[step=5, black, very thin,xshift=\xx cm, yshift=\yy cm] (0,5) grid +(5,5);
	\draw[step=5, black, very thin,xshift=\xx cm, yshift=\yy cm] (5,5) grid +(5,5);
	
	\draw[step=1, gray, very thin,xshift=\xx cm, yshift=\yy cm] (0,0) grid +(5,5);
	\draw[step=1, gray, very thin,xshift=\xx cm, yshift=\yy cm] (0,5) grid +(5,5);
	\draw[step=1, gray, very thin,xshift=\xx cm, yshift=\yy cm] (5,5) grid +(5,5);

	\foreach \x in {0,...,4}{
		\foreach \y in {0,...,4}{
			\draw (\xx+\x+0.5,\yy+\y+0.5) node[black] {$0$};
		}
	}
	
	\foreach \x in {0,...,4}{
		\foreach \y in {5}{
			\draw (\xx+\x+0.5,\yy+\y+0.5) node[black] {$0$};
		}
	}

	\draw (\xx+0.5,\yy+9.5) node[black] {$e$};
	\draw (\xx+1.5,\yy+7.5) node[black] {$e$};
	\draw (\xx+2.5,\yy+8.5) node[black] {$e$};
	\draw (\xx+4.5,\yy+6.5) node[black] {$e$};
	
	\draw (\xx+5.5,\yy+7.5) node[black] {$e$};
	\draw (\xx+6.5,\yy+5.5) node[black] {$e$};
	\draw (\xx+7.5,\yy+6.5) node[black] {$e$};
	\draw (\xx+8.5,\yy+8.5) node[black] {$e$};
	\draw (\xx+9.5,\yy+9.5) node[black] {$e$};

	\draw[->]  (11,-21) to (13,-21);

	\def\xx{14}
	\def\yy{-26}

	\draw[step=5, black, very thin,xshift=\xx cm, yshift=\yy cm] (0,0) grid +(5,5);
	\draw[step=5, black, very thin,xshift=\xx cm, yshift=\yy cm] (0,5) grid +(5,5);
	\draw[step=5, black, very thin,xshift=\xx cm, yshift=\yy cm] (5,5) grid +(5,5);
	
	\draw[step=1, gray, very thin,xshift=\xx cm, yshift=\yy cm] (0,0) grid +(5,5);
	\draw[step=1, gray, very thin,xshift=\xx cm, yshift=\yy cm] (0,5) grid +(5,5);
	\draw[step=1, gray, very thin,xshift=\xx cm, yshift=\yy cm] (5,5) grid +(5,5);

	\foreach \x in {0,...,4}{
		\foreach \y in {0,...,4}{
			\draw (\xx+\x+0.5,\yy+\y+0.5) node[black] {$0$};
		}
	}
	
	\foreach \x in {0,...,4}{
		\foreach \y in {5,...,6}{
			\draw (\xx+\x+0.5,\yy+\y+0.5) node[black] {$0$};
		}
	}		

	\draw (\xx+0.5,\yy+9.5) node[black] {$e$};
	\draw (\xx+1.5,\yy+7.5) node[black] {$e$};
	\draw (\xx+2.5,\yy+8.5) node[black] {$e$};
	
	\draw (\xx+5.5,\yy+7.5) node[black] {$e$};
	\draw (\xx+6.5,\yy+5.5) node[black] {$e$};
	\draw (\xx+7.5,\yy+6.5) node[black] {$e$};
	\draw (\xx+8.5,\yy+8.5) node[black] {$e$};
	\draw (\xx+9.5,\yy+9.5) node[black] {$e$};

	\draw[->]  (25,-21) to (27,-21);

	\def\xx{28}
	\def\yy{-26}

	\draw[step=5, black, very thin,xshift=\xx cm, yshift=\yy cm] (0,0) grid +(5,5);
	\draw[step=5, black, very thin,xshift=\xx cm, yshift=\yy cm] (0,5) grid +(5,5);
	\draw[step=5, black, very thin,xshift=\xx cm, yshift=\yy cm] (5,5) grid +(5,5);
	
	\draw[step=1, gray, very thin,xshift=\xx cm, yshift=\yy cm] (0,0) grid +(5,5);
	\draw[step=1, gray, very thin,xshift=\xx cm, yshift=\yy cm] (0,5) grid +(5,5);
	\draw[step=1, gray, very thin,xshift=\xx cm, yshift=\yy cm] (5,5) grid +(5,5);

	\foreach \x in {0,...,4}{
		\foreach \y in {0,...,4}{
			\draw (\xx+\x+0.5,\yy+\y+0.5) node[black] {$0$};
		}
	}
	
	\foreach \x in {0,...,4}{
		\foreach \y in {5,...,9}{
			\draw (\xx+\x+0.5,\yy+\y+0.5) node[black] {$0$};
		}
	}		
	
	\draw (\xx+5.5,\yy+7.5) node[black] {$e$};
	\draw (\xx+6.5,\yy+5.5) node[black] {$e$};
	\draw (\xx+7.5,\yy+6.5) node[black] {$e$};
	\draw (\xx+8.5,\yy+8.5) node[black] {$e$};
	\draw (\xx+9.5,\yy+9.5) node[black] {$e$};

	\draw[->]  (39,-21) to (41,-21);

	\def\xx{42}
	\def\yy{-26}

	\draw[step=5, black, very thin,xshift=\xx cm, yshift=\yy cm] (0,0) grid +(5,5);
	\draw[step=5, black, very thin,xshift=\xx cm, yshift=\yy cm] (0,5) grid +(5,5);
	\draw[step=5, black, very thin,xshift=\xx cm, yshift=\yy cm] (5,5) grid +(5,5);
	
	\draw[step=1, gray, very thin,xshift=\xx cm, yshift=\yy cm] (0,0) grid +(5,5);
	\draw[step=1, gray, very thin,xshift=\xx cm, yshift=\yy cm] (0,5) grid +(5,5);
	\draw[step=1, gray, very thin,xshift=\xx cm, yshift=\yy cm] (5,5) grid +(5,5);

	\foreach \x in {0,...,4}{
		\foreach \y in {0,...,4}{
			\draw (\xx+\x+0.5,\yy+\y+0.5) node[black] {$0$};
		}
	}
	
	\foreach \x in {0,...,4}{
		\foreach \y in {5,...,9}{
			\draw (\xx+\x+0.5,\yy+\y+0.5) node[black] {$0$};
		}
	}		
	
	\foreach \x in {5,...,9}{
		\foreach \y in {5,...,9}{
			\draw (\xx+\x+0.5,\yy+\y+0.5) node[black] {$0$};
		}
	}
	
\end{tikzpicture}

\caption{\label{fig:bp12}First steps of propagating the empty cluster. $0$
represents an empty site, otherwise the state is the initial one.
$e$ stands for an easy site.}
\end{figure}
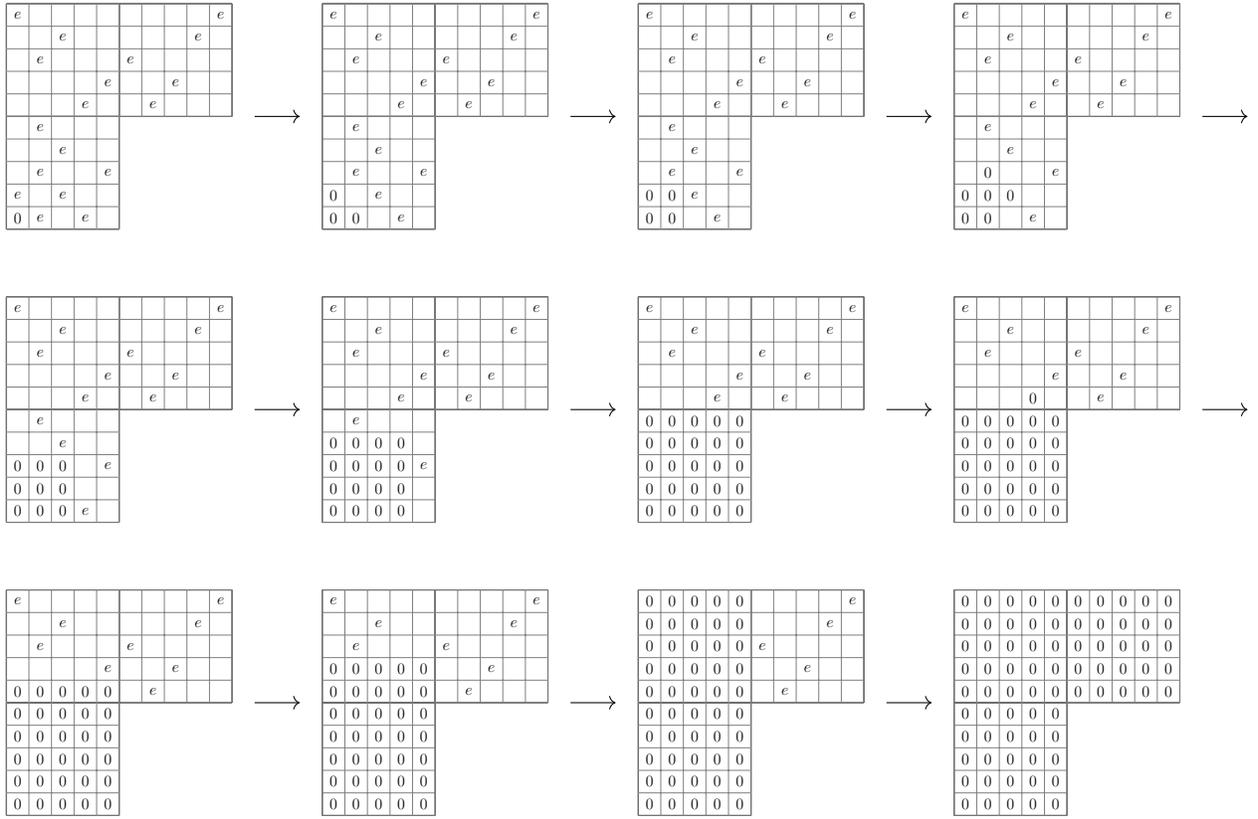
The next two claims will show how a cluster of empty sites could propagate.
See \figref{bp12}.
\begin{claim}
\label{claim:infectingexcellentsquares}Assume that $\left[L\right]^{2}$
is excellent, and that $\left(1,1\right)$ is initially empty. Then
$\left[L\right]^{2}$ will be entirely emptied by time $L^{2}$.
\end{claim}

\begin{proof}
This could be done by induction on the size of the empty square \textendash{}
assume that $\left[l\right]^{2}$ is entirely emptied for some $l\le L$.
By the definition of an excellent square, there is an easy site $x\in\left\{ l+1\right\} \times\left[l\right]$.
Its neighbor to the left is empty (since it is in $\left[l\right]^{2}$),
so at the next time step this site will also be empty. Once $x$ is
empty, the two sites $x\pm e_{2}$ could be emptied, and then the
sites $x\pm2e_{2}$ and so on, as long as they stay in $\left\{ l+1\right\} \times\left[l\right]$.
Thus, at time $l$ all sites in $\left\{ l+1\right\} \times\left[l\right]$
will be empty, and by the same reasoning the sites of $\left[l\right]\times\left\{ l+1\right\} $
will also be empty. Since $\left(l+1,l+1\right)$ has two empty neighbors
it will be emptied at step $l+1$, and thus $\left[l+1\right]^{2}$
will be emptied.
\end{proof}
\begin{claim}
\label{claim:infectionpropagatesforfa2}Assume that $\left[L\right]^{2}$
is good, and that it has a neighboring square that is entirely empty
by time $T$. Then $\left[L\right]^{2}$ will be entirely empty by
time $T+L^{2}$.
\end{claim}

\begin{proof}
We can empty $\left[L\right]^{2}$ line by line (or column by column,
depending on whether its empty neighbor is in the horizontal direction
or the vertical one). For each line, we start by emptying the easy
site that it contains, and then continue to propagate.
\end{proof}
\begin{defn}
Until the end of the proof of the upper bound, $L$ will be the minimal
length for which the probability to be good exceeds $p^{\text{SP}}+0.01$.
\end{defn}

\begin{defn}
$\mathcal{C}$ will denote the infinite cluster of good boxes of the
form $Li+\left[L\right]^{2}$ for $i\in\zz^{2}$. $\mathcal{C}_{0}$
will denote the cluster of the origin surrounded by a path in $\mathcal{C}$,
or just the origin if it is in $\mathcal{C}$. $\partial\mathcal{C}_{0}$
will be the outer boundary of $\mathcal{C}_{0}$ (namely the boxes
out of $\mathcal{C}$ that have a neighbor in $\mathcal{C}_{0}$).
Note that $\mathcal{C}_{0}$ is finite and that $\partial\mathcal{C}_{0}$
is connected.
\end{defn}

\begin{claim}
\label{claim:infectionreachestheorigin}Assume that at time $T$ one
of the boxes on $\partial\mathcal{C}_{0}$ is entirely empty. Then
by time $T+T_{0}$ the origin will be empty, where $T_{0}=\left(\left|\partial\mathcal{C}_{0}\right|+\left|\mathcal{C}_{0}\right|\right)L^{2}$.
\end{claim}

\begin{proof}
By \claimref{infectionpropagatesforfa2}, the boundary $\partial\mathcal{C}_{0}$
will be emptied by time $T_{0}+L^{2}\left|\partial\mathcal{C}_{0}\right|$.
Then, at each time step at least one site of $\mathcal{C}_{0}$ must
be emptied, since no finite region could stay occupied forever.
\end{proof}
\begin{claim}
Assume that a box $Li+\left[L\right]^{2}$ in $\mathcal{C}$ is empty
at time $T$. Also, assume that the graph distance in $\mathcal{C}$
between this box and $\partial\mathcal{C}_{0}$ is $l$. Then by time
$T+lL^{2}+T_{0}$ the origin will be empty.
\end{claim}

\begin{proof}
This is again a direct application of claims \ref{claim:infectionpropagatesforfa2}
and \ref{claim:infectionreachestheorigin}.
\end{proof}
Finally, we will use the following result from percolation theory:
\begin{claim}
For $l$ large enough, the number of boxes in $\mathcal{C}$ that
are at graph distance in $\mathcal{C}$ at most $l$ from $\partial\mathcal{C}_{0}$
is greater than $\theta l^{2}$, where $\theta$ depends only on the
probability that a box is good.
\end{claim}

\begin{proof}
By ergodicity the cluster $\mathcal{C}$ has an almost sure positive
density, so in particular
\[
\liminf_{l\rightarrow\infty}\frac{\left|\mathcal{C}\cap\left[-l,l\right]^{2}\right|}{\left|\left[-l,l\right]^{2}\right|}>0.
\]
By \cite{AntalPisztora1996chemicaldistancesupercritical}, there exists
a positive constant $\rho$ such that boxes of graph distance $l$
from the origin must be in the box $\left[-\frac{1}{\rho}l,\frac{1}{\rho}l\right]^{2}$
for $l$ large enough. Combining these two facts proves the claim.
\end{proof}
This claim together with a large deviation estimate yields
\begin{cor}
\label{cor:densityofexcellentboxes}For $l$ large enough, the number
of excellent boxes in $\mathcal{C}$ that are at graph distance in
$\mathcal{C}$ at most $l$ from $\partial\mathcal{C}_{0}$ is greater
than $\theta^{\prime}l^{2}$, where $\theta^{\prime}=0.99\,\theta p_{L}$.
\end{cor}

We can now put all the ingredients together and obtain the upper bound.

Fix $c>0$, and $l=\frac{c}{\sqrt{q}}$. By \corref{densityofexcellentboxes},
for $q$ small, there are at least $\frac{\theta^{\prime}c^{2}}{q}$
excellent boxes at distance smaller than $l$ from $\partial\mathcal{C}_{0}$.
If one of them contains an empty site at the bottom left corner, the
origin will be emptied by time $\left(l+1\right)L^{2}+T_{0}$. For
$q$ small enough, this time is bounded by $\frac{2cL^{2}}{\sqrt{q}}$.
The probability that non of them do is $\left(1-q\right)^{\frac{\theta^{\prime}c^{2}}{q}}$,
thus it tends to $0$ uniformly in $q$ as $c\rightarrow\infty$. 

\subsubsection{Proof of \eqref{fa12d2bplowerbound}}

The lower bound results from the simple observation, that the root
could only be infected by time $t$ if there is an empty site at distance
smaller than $t$. The probability of that event is $1-\left(1-q\right)^{4t^{2}}$,
and taking $t=\frac{a}{\sqrt{q}}$ and $q$ small enough this probability
is bounded by $1-e^{-2a}$. This tends to $0$ with $a$ uniformly
in $q$, which finishes the proof.

\subsection{Mixed threshold KCM on $\protect\zz^{2}$}

\subsubsection{Spectral gap}

The spectral gap of this model is dominated by that of the FA2f model.
Fix any $\gamma$ strictly greater than the gap of FA2f. Then there
is a local non-constant function $f$ such that
\[
\frac{\mathcal{D}^{\text{FA2f}}f}{\text{Var}f}\le\gamma,
\]
where $\mathcal{D}^{\text{FA2f}}$ is the Dirichlet form of the FA2f
model.

Since $f$ is local, it is supported in some square of size $L\times L$,
for $L$ big enough. $\nu$-almost surely it is possible to find a
far away square in $\zz^{2}$ of size $L\times L$ that contains only
difficult sites. By translation invariance of the FA2f model we can
assume that this is the square in which $f$ is supported. In this
case, $\mathcal{D}f=\mathcal{D}^{\text{FA2f}}f$, and this shows that
indeed the gap of the model with random threshold is smaller than
that of FA2f, which by \cite{cancrini2008kcmzoo} is bounded by $e^{-\nicefrac{c}{q}}$.

\subsubsection{Proof of \eqref{fa12d2kcmupperbound}}

In this part we will use \corref{dirichletequalsexpectation} in order
to bound $\tau_{0}$ by a path argument. As in the proof of the upper
bound for the bootstrap percolation, we will consider the good squares
(see \defref{goodsquareforfa12}) and their infinite cluster. In fact,
by \claimref{squaresaregood}, by choosing $L$ big enough we may
assume that the box $\left[L\right]^{2}$ is in this cluster. Let
us fix this $L$ until the end of this part. We will also choose an
infinite self avoiding path of good boxes starting at the origin and
denote it by $i_{0},i_{1},i_{2},\dots$. Note that this path depends
on $\omega$ but not on $\eta$.

On this cluster empty sites will be able to propagate, and the next
definition will describe the seed needed in order to start this propagation.
\begin{defn}
A box in $\zz^{2}$ is \textit{essentially empty} if it is good and
contains an entire line or an entire column of empty sites. This will
depend on both $\omega$ and $\eta$.
\end{defn}

In order to guarantee the presence of an essentially empty box we
will fix $l=q^{-L-1}$, and define the bad event
\begin{defn}
$B=\left\{ \text{none of the boxes }i_{0},\dots,i_{l}\text{ is essentially empty}\right\} $.
For fixed $\omega$ the path $i_{0},i_{1},i_{2},\dots$ is fixed,
and $B$ is an event in $\Omega$.
\end{defn}

A simple bound shows that
\begin{equation}
\mu\left(B\right)\le\left(1-q^{L}\right)^{l}\le e^{-\nicefrac{1}{q}}.\label{eq:badeventhassmallprob}
\end{equation}

We can use this bound in order to bound the hitting time at $B$:
\begin{claim}
\label{claim:taubisbig}There exists $C>0$ such that $\pp_{\mu}\left(\tau_{B}\le t\right)\le Ce^{-\nicefrac{1}{q}}\,t$.
\end{claim}

\begin{proof}
We use the graphical construction of the Markov process. In order
to hit $B$, we must hit it at a certain clock ring taking place in
one of the sites of $\cup_{n=1}^{l}\left(Li+\left[L\right]^{2}\right)$.
Therefore,
\begin{align*}
\pp\left(\tau_{B}\le t\right) & \le\mathbb{P}\left[\text{more than }2\left(2L+1\right)^{2}lt\text{ rings by time }t\right]+2\left(2L+1\right)^{2}l\,t\,\mu\left(B\right)\\
 & \le e^{-\left(2L+1\right)^{2}q^{-L-1}t}+2\left(2L+1\right)^{2}q^{-L-1}\,t\,e^{-\nicefrac{1}{q}}\le Ce^{-\nicefrac{1}{q}}\,t.
\end{align*}
\end{proof}
In order to bound $\tau_{0}$ we will study the hitting time of $A=B\cup\left\{ \eta\left(0\right)=0\right\} $.
\begin{lem}
\label{lem:fa12path}Fix $\eta\in\Omega$. Then there exists a path
$\eta_{0},\dots,\eta_{N}$ of configurations and a sequence of sites
$x_{0},\dots x_{N-1}$ such that
\begin{enumerate}
\item $\eta_{0}=\eta$,
\item $\eta_{N}\in A$,
\item $\eta_{i+1}=\eta_{i}^{x_{i}}$,
\item $c_{x_{i}}\left(\eta_{i}\right)=1$,
\item $N\le4L^{2}l$,
\item For all $i\le N$, $\eta_{i}$ differs from $\eta$ at at most $3L$
points, contained in at most two neighboring boxes.
\end{enumerate}
\end{lem}

\begin{proof}
If $\eta\in A$, we take the path $\eta$ with $N=0$. Otherwise $\eta\in B^{c}$,
so there is an essentially empty box in $i_{0},\dots,i_{l}$, which
contain an empty column (or row). We can then create an empty column
(row) next to it and propagate that column (row) as in \figref{propagationgacolumn12}.
When the path rotates we can rotate this propagating column (row)
as show in \figref{rotatingcol12}.
\end{proof}
\begin{figure}
\begin{tikzpicture}[scale=0.3, every node/.style={scale=0.6}]
	\def\xx{0}
	\def\yy{0}
	
	\draw[step=5, black, very thin,xshift=\xx cm, yshift=\yy cm] (0,0) grid +(3,5);
	\draw[step=1, gray, very thin,xshift=\xx cm, yshift=\yy cm] (0,0) grid +(3,5);

	\foreach \y in {0,...,4}{
		\draw (\xx+0.5,\yy+\y+0.5) node[black] {$0$};
	}
	
	\draw (\xx+1.5,\yy+1.5) node[black] {$e$};
	\draw (\xx+2.5,\yy+2.5) node[black] {$e$};

	\draw[->]  (3.5,2.5) to (5,2.5);

	\def\xx{5.5}
	\def\yy{0}
	
	\draw[step=5, black, very thin,xshift=\xx cm, yshift=\yy cm] (0,0) grid +(3,5);
	\draw[step=1, gray, very thin,xshift=\xx cm, yshift=\yy cm] (0,0) grid +(3,5);

	\foreach \y in {0,...,4}{
		\draw (\xx+0.5,\yy+\y+0.5) node[black] {$0$};
	}

	\draw (\xx+1.5,\yy+1.5) node[black] {$0$};
	\draw (\xx+2.5,\yy+2.5) node[black] {$e$};

	\draw[->]  (9,2.5) to (10.5,2.5);

	\def\xx{11}
	\def\yy{0}
	
	\draw[step=5, black, very thin,xshift=\xx cm, yshift=\yy cm] (0,0) grid +(3,5);
	\draw[step=1, gray, very thin,xshift=\xx cm, yshift=\yy cm] (0,0) grid +(3,5);

	\foreach \y in {0,...,4}{
		\draw (\xx+0.5,\yy+\y+0.5) node[black] {$0$};
	}
	\foreach \y in {0,...,4}{
		\draw (\xx+1.5,\yy+\y+0.5) node[black] {$0$};
	}

	\draw (\xx+2.5,\yy+2.5) node[black] {$e$};

	\draw[->]  (14.5,2.5) to (16,2.5);

	\def\xx{16.5}
	\def\yy{0}
	
	\draw[step=5, black, very thin,xshift=\xx cm, yshift=\yy cm] (0,0) grid +(3,5);
	\draw[step=1, gray, very thin,xshift=\xx cm, yshift=\yy cm] (0,0) grid +(3,5);

	\foreach \y in {0,...,4}{
		\draw (\xx+0.5,\yy+\y+0.5) node[black] {$0$};
	}
	\foreach \y in {0,...,4}{
		\draw (\xx+1.5,\yy+\y+0.5) node[black] {$0$};
	}
	
	\foreach \y in {0,...,4}{
		\draw (\xx+2.5,\yy+\y+0.5) node[black] {$0$};
	}

	\draw[->]  (20,2.5) to (21.5,2.5);

	\def\xx{22}
	\def\yy{0}
	
	\draw[step=5, black, very thin,xshift=\xx cm, yshift=\yy cm] (0,0) grid +(3,5);
	\draw[step=1, gray, very thin,xshift=\xx cm, yshift=\yy cm] (0,0) grid +(3,5);

	\foreach \y in {0,...,4}{
		\draw (\xx+0.5,\yy+\y+0.5) node[black] {$0$};
	}
	\foreach \y in {0,...,3}{
		\draw (\xx+1.5,\yy+\y+0.5) node[black] {$0$};
	}
	
	\foreach \y in {0,...,4}{
		\draw (\xx+2.5,\yy+\y+0.5) node[black] {$0$};
	}

	\draw[->]  (25.5,2.5) to (27,2.5);

	\def\xx{27.5}
	\def\yy{0}
	
	\draw[step=5, black, very thin,xshift=\xx cm, yshift=\yy cm] (0,0) grid +(3,5);
	\draw[step=1, gray, very thin,xshift=\xx cm, yshift=\yy cm] (0,0) grid +(3,5);

	\foreach \y in {0,...,4}{
		\draw (\xx+0.5,\yy+\y+0.5) node[black] {$0$};
	}
	
	\foreach \y in {0,...,4}{
		\draw (\xx+2.5,\yy+\y+0.5) node[black] {$0$};
	}
	
	\draw (\xx+1.5,\yy+1.5) node[black] {$0$};

	\draw[->]  (31,2.5) to (32.5,2.5);

	\def\xx{33}
	\def\yy{0}
	
	\draw[step=5, black, very thin,xshift=\xx cm, yshift=\yy cm] (0,0) grid +(3,5);
	\draw[step=1, gray, very thin,xshift=\xx cm, yshift=\yy cm] (0,0) grid +(3,5);

	\foreach \y in {0,...,4}{
		\draw (\xx+0.5,\yy+\y+0.5) node[black] {$0$};
	}
	
	\foreach \y in {0,...,4}{
		\draw (\xx+2.5,\yy+\y+0.5) node[black] {$0$};
	}
	
	\draw (\xx+1.5,\yy+1.5) node[black] {$e$};
	
\end{tikzpicture}

\caption{\label{fig:propagationgacolumn12}Creating an empty column and propagating
it using the easy sites. $0$ represents an empty site, otherwise
the state is the initial one. $e$ stands for an easy site.}

\end{figure}
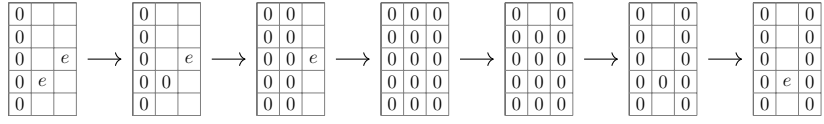

\begin{figure}
\begin{tikzpicture}[scale=0.3, every node/.style={scale=0.6}]
	\def\xx{0}
	\def\yy{0}
	
	\draw[step=5, black, very thin,xshift=\xx cm, yshift=\yy cm] (0,0) grid +(5,5);
	\draw[step=1, gray, very thin,xshift=\xx cm, yshift=\yy cm] (0,0) grid +(5,5);

	\foreach \y in {0,...,4}{
		\draw (\xx+0.5,\yy+\y+0.5) node[black] {$0$};
	}
	
	\draw (\xx+1.5,\yy+1.5) node[black] {$e$};
	\draw (\xx+2.5,\yy+2.5) node[black] {$e$};
	\draw (\xx+3.5,\yy+4.5) node[black] {$e$};
	\draw (\xx+3.5,\yy+0.5) node[black] {$e$};
	\draw (\xx+4.5,\yy+3.5) node[black] {$e$};

	\draw[->]  (6,2.5) to (8,2.5);

	\def\xx{9}
	\def\yy{0}
	
	\draw[step=5, black, very thin,xshift=\xx cm, yshift=\yy cm] (0,0) grid +(5,5);
	\draw[step=1, gray, very thin,xshift=\xx cm, yshift=\yy cm] (0,0) grid +(5,5);

	\foreach \y in {0,...,4}{
		\draw (\xx+0.5,\yy+\y+0.5) node[black] {$0$};
	}
	\foreach \y in {0,...,4}{
		\draw (\xx+1.5,\yy+\y+0.5) node[black] {$0$};
	}
	\foreach \y in {0,...,4}{
		\draw (\xx+2.5,\yy+\y+0.5) node[black] {$0$};
	}
	
	\draw (\xx+3.5,\yy+4.5) node[black] {$e$};
	\draw (\xx+3.5,\yy+0.5) node[black] {$e$};
	\draw (\xx+4.5,\yy+3.5) node[black] {$e$};

	\draw[->]  (15,2.5) to (17,2.5);

	\def\xx{18}
	\def\yy{0}
	
	\draw[step=5, black, very thin,xshift=\xx cm, yshift=\yy cm] (0,0) grid +(5,5);
	\draw[step=1, gray, very thin,xshift=\xx cm, yshift=\yy cm] (0,0) grid +(5,5);

	\foreach \y in {0,...,4}{
		\draw (\xx+0.5,\yy+\y+0.5) node[black] {$0$};
	}

	\foreach \y in {0,...,4}{
		\draw (\xx+2.5,\yy+\y+0.5) node[black] {$0$};
	}
	
	\draw (\xx+1.5,\yy+1.5) node[black] {$e$};
	\draw (\xx+1.5,\yy+4.5) node[black] {$0$};
	\draw (\xx+3.5,\yy+4.5) node[black] {$e$};
	\draw (\xx+3.5,\yy+0.5) node[black] {$e$};
	\draw (\xx+4.5,\yy+3.5) node[black] {$e$};

	\draw[->]  (24,2.5) to (26,2.5);

	\def\xx{27}
	\def\yy{0}
	
	\draw[step=5, black, very thin,xshift=\xx cm, yshift=\yy cm] (0,0) grid +(5,5);
	\draw[step=1, gray, very thin,xshift=\xx cm, yshift=\yy cm] (0,0) grid +(5,5);

	\foreach \y in {0,...,4}{
		\draw (\xx+0.5,\yy+\y+0.5) node[black] {$0$};
	}

	\foreach \y in {0,...,4}{
		\draw (\xx+4.5,\yy+\y+0.5) node[black] {$0$};
	}

	\draw (\xx+1.5,\yy+4.5) node[black] {$0$};
	\draw (\xx+2.5,\yy+4.5) node[black] {$0$};
	\draw (\xx+3.5,\yy+4.5) node[black] {$0$};
	
	\draw (\xx+1.5,\yy+1.5) node[black] {$e$};
	\draw (\xx+3.5,\yy+0.5) node[black] {$e$};
	\draw (\xx+2.5,\yy+2.5) node[black] {$e$};

	\draw[->]  (33,2.5) to (35,2.5);

	\def\xx{36}
	\def\yy{0}
	
	\draw[step=5, black, very thin,xshift=\xx cm, yshift=\yy cm] (0,0) grid +(5,5);
	\draw[step=1, gray, very thin,xshift=\xx cm, yshift=\yy cm] (0,0) grid +(5,5);

	\foreach \y in {0,...,4}{
		\draw (\xx+0.5,\yy+\y+0.5) node[black] {$0$};
	}

	\foreach \y in {0,...,4}{
		\draw (\xx+3.5,\yy+\y+0.5) node[black] {$0$};
	}

	\draw (\xx+1.5,\yy+4.5) node[black] {$0$};
	\draw (\xx+2.5,\yy+4.5) node[black] {$0$};
	\draw (\xx+4.5,\yy+4.5) node[black] {$0$};
	
	\draw (\xx+1.5,\yy+1.5) node[black] {$e$};
	\draw (\xx+4.5,\yy+3.5) node[black] {$e$};
	\draw (\xx+2.5,\yy+2.5) node[black] {$e$};

	\draw[->]  (42,2.5) to (44,2.5);

	\def\xx{0}
	\def\yy{-7}

	\draw[step=5, black, very thin,xshift=\xx cm, yshift=\yy cm] (0,0) grid +(5,5);
	\draw[step=1, gray, very thin,xshift=\xx cm, yshift=\yy cm] (0,0) grid +(5,5);

	\foreach \y in {0,...,4}{
		\draw (\xx+0.5,\yy+\y+0.5) node[black] {$0$};
	}
	\foreach \x in {1,...,4}{
		\draw (\xx+\x+0.5,\yy+4.5) node[black] {$0$};
	}
	
	\draw (\xx+1.5,\yy+1.5) node[black] {$e$};
	\draw (\xx+2.5,\yy+2.5) node[black] {$e$};
	\draw (\xx+3.5,\yy+0.5) node[black] {$e$};
	\draw (\xx+4.5,\yy+3.5) node[black] {$e$};

	\draw[->]  (6,-4.5) to (8,-4.5);

	\def\xx{9}
	\def\yy{-7}

	\draw[step=5, black, very thin,xshift=\xx cm, yshift=\yy cm] (0,0) grid +(5,5);
	\draw[step=1, gray, very thin,xshift=\xx cm, yshift=\yy cm] (0,0) grid +(5,5);

	\foreach \y in {0,...,4}{
		\draw (\xx+0.5,\yy+\y+0.5) node[black] {$0$};
	}
	\foreach \x in {1,...,4}{
		\draw (\xx+\x+0.5,\yy+0.5) node[black] {$0$};
	}
	
	\draw (\xx+1.5,\yy+1.5) node[black] {$e$};
	\draw (\xx+2.5,\yy+2.5) node[black] {$e$};
	\draw (\xx+3.5,\yy+4.5) node[black] {$e$};
	\draw (\xx+4.5,\yy+3.5) node[black] {$e$};

	\draw[->]  (15,-4.5) to (17,-4.5);

	\def\xx{18}
	\def\yy{-7}

	\draw[step=5, black, very thin,xshift=\xx cm, yshift=\yy cm] (0,0) grid +(5,5);
	\draw[step=1, gray, very thin,xshift=\xx cm, yshift=\yy cm] (0,0) grid +(5,5);

	\foreach \y in {1,...,4}{
		\draw (\xx+0.5,\yy+\y+0.5) node[black] {$0$};
	}
	\foreach \x in {1,...,4}{
		\draw (\xx+\x+0.5,\yy+1.5) node[black] {$0$};
	}
	
	\draw (\xx+3.5,\yy+0.5) node[black] {$e$};
	\draw (\xx+2.5,\yy+2.5) node[black] {$e$};
	\draw (\xx+3.5,\yy+4.5) node[black] {$e$};
	\draw (\xx+4.5,\yy+3.5) node[black] {$e$};

	\draw[->]  (24,-4.5) to (26,-4.5);

	\def\xx{27}
	\def\yy{-7}

	\draw[step=5, black, very thin,xshift=\xx cm, yshift=\yy cm] (0,0) grid +(5,5);
	\draw[step=1, gray, very thin,xshift=\xx cm, yshift=\yy cm] (0,0) grid +(5,5);

	\foreach \x in {0,...,4}{
		\draw (\xx+\x+0.5,\yy+4.5) node[black] {$0$};
	}

	\draw (\xx+3.5,\yy+0.5) node[black] {$e$};
	\draw (\xx+2.5,\yy+2.5) node[black] {$e$};
	\draw (\xx+1.5,\yy+1.5) node[black] {$e$};
	\draw (\xx+4.5,\yy+3.5) node[black] {$e$};
\end{tikzpicture}

\caption{\label{fig:rotatingcol12}Rotating an empty column in a good box.
$0$ represents an empty site, otherwise the state is the initial
one. $e$ stands for an easy site.}

\end{figure}
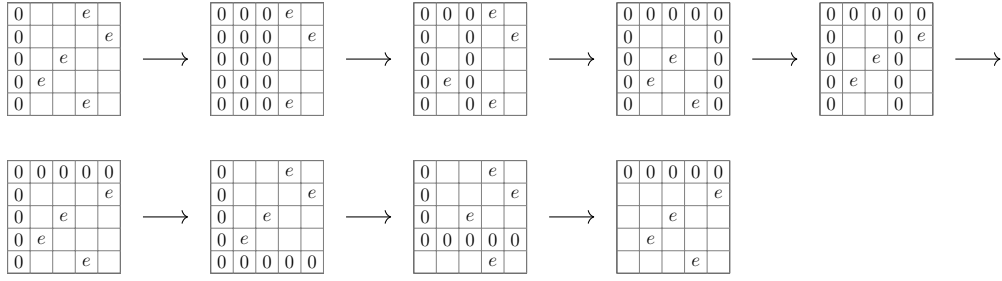

We can use this path together with \corref{dirichletequalsexpectation}
in order to bound $\tau_{A}$.
\begin{lem}
\label{lem:upperboundontaul}There exists $C_{L}>0$ (that may depends
on $L$ but not on $q$) such that $\mu\left(\tau_{A}\right)\le C_{L}\,q^{-5L-2}$.
\end{lem}

\begin{proof}
Since $\tau$ vanishes on $A$, taking the path defined in \lemref{fa12path},
\[
\tau_{A}\left(\eta\right)=\sum_{i=0}^{N-1}\left(\tau_{A}\left(\eta_{i}\right)-\tau_{A}\left(\eta_{i+1}\right)\right).
\]
In the following we use the notation
\[
\nabla_{x}\tau_{A}\left(\eta\right)=\tau_{A}\left(\eta\right)-\tau_{A}\left(\eta^{x}\right).
\]

Then, by Cauchy Schwartz inequality,
\begin{eqnarray*}
\mu\left(\tau_{A}\right)^{2} & \le & \mu\left(\tau_{A}^{2}\right)=\sum_{\eta}\mu\left(\eta\right)\left(\sum_{i=0}^{N-1}\nabla_{x_{i}}\tau_{A}\left(\eta_{i}\right)\right)^{2}\\
 & \le & \sum_{\eta}\mu\left(\eta\right)N\sum_{i}c_{x_{i}}\left(\eta_{i}\right)\left(\nabla_{x_{i}}\tau_{A}\left(\eta_{i}\right)\right)^{2}\\
 & = & \sum_{\eta}\mu\left(\eta\right)N\sum_{i}\sum_{z}\sum_{\eta^{\prime}}c_{z}\left(\eta^{\prime}\right)\left(\nabla_{z}\tau_{A}\left(\eta^{\prime}\right)\right)^{2}\One_{z=x_{i}}\One_{\eta^{\prime}=\eta_{i}}.
\end{eqnarray*}
By property number 6 of the path, we know that $\mu\left(\eta\right)\le q^{-3L}\mu\left(\eta^{\prime}\right)$,
so we obtain

\[
\mu\left(\tau_{A}\right)^{2}\le q^{-3L}N\sum_{\eta^{\prime}}\mu\left(\eta^{\prime}\right)\sum_{z}c_{z}\left(\eta^{\prime}\right)\left(\nabla_{z}\tau_{A}\left(\eta^{\prime}\right)\right)^{2}\sum_{i}\One_{z=x_{i}}\sum_{\eta}\One_{\eta^{\prime}=\eta_{i}}.
\]
Still using property $6$, $\eta$ differs form $\eta^{\prime}$ at
at most $3L$ point, all of them in the box containing $z$ or in
one of the two neighboring boxes. This gives the bound $\sum_{\eta}\One_{\eta^{\prime}=\eta^{(i)}}\le\left(3L^{2}\right)^{3L}$.
Finally, bounding $\One_{z=x_{i}}$ by $1$,
\begin{align*}
\mu\left(\tau_{A}\right)^{2} & \le q^{-3L}\left(3L^{2}\right)^{3L}N^{2}\sum_{\eta^{\prime}}\mu\left(\eta^{\prime}\right)\sum_{z}c_{z}\left(\eta^{\prime}\right)\left(\nabla_{z}\tau_{A}\left(\eta^{\prime}\right)\right)^{2}\\
 & \le16\left(3L^{2}\right)^{3L}L^{4}q^{-5L-2}\,\mathcal{D}\tau_{A}.
\end{align*}
This concludes the proof of the lemma by \corref{dirichletequalsexpectation}.
\end{proof}
Using this lemma and the bound on $\tau_{B}$ in \claimref{taubisbig},
we can finish the estimation of the upper bound. By Markov inequality
\[
\mu\left(\tau_{A}\ge C_{L}\,q^{-5L-3}\right)\le q.
\]
On the other hand, by \claimref{taubisbig},
\begin{align*}
\mu\left(\tau_{A}<C_{L}\,q^{-5L-5}\right) & \le\mu\left(\tau_{0}<C_{L}\,q^{-5L-3}\right)+\mu\left(\tau_{B}<C_{L}\,q^{-5L-3}\right)\\
 & \le\mu\left(\tau_{0}<C_{L}\,q^{-5L-3}\right)+C_{L}^{\prime}e^{-\nicefrac{1}{q}}.
\end{align*}
Therefore
\[
\mu\left(\tau_{0}\ge C_{L}\,q^{-5L-3}\right)\le q+C_{L}^{\prime}e^{-\nicefrac{1}{q}},
\]
and taking $\overline{\alpha}=5L+3$ will suffice.

Concerning \remref{exponentsaretruelyrandom}, fix $L_{0}$ such that
the probability that $\left[L_{0}\right]^{2}$ is good exceeds $p^{\text{SP}}$.
Then, for $\alpha_{0}=5L_{0}+3$, $\nu\left(\overline{\alpha}\le\alpha_{0}\right)>0$.
We will see later that the other inequality holds as well for the
same $\alpha_{0}$.

\subsubsection{Proof of \eqref{fa12d2kcmlowerbound}}

A trivial bound could be obtained by taking any $\underline{\alpha}<1$,
since the rate at which the origin becomes empty is always at most
$q$. We will, however, look for a bound that will describe better
the effect of the disorder, and will allow us to prove \remref{exponentsaretruelyrandom}.
The basic observation for the estimation of this lower bound is that
if $\left[-L,L\right]^{2}$ is initially occupied, and if it contains
only difficult sites, then at some point we will need to empty at
least $\frac{L}{2}$ sites before the origin could be emptied. This
energy barrier forces $\tau_{A}$ to be greater than $q^{-\nicefrac{L}{2}}$.

In the following we will fix $L$ such that $\left[-L,L\right]^{2}$
contains only difficult sites (so it is not the same $L$ we have
used for the upper bound).
\begin{defn}
For a rectangle $R$, the span of $R$ are the sites that could be
emptied with the bootstrap percolation using only the $0$s of $R$.
If the span of $R$ equals $R$ we say that $R$ is internally spanned.
\end{defn}

The next fact is a consequence of the fact a set which is stable under
the bootstrap percolation must be a rectangle \cite{aizenmanlebowitz1988metastabilityz2}.
\begin{fact}
\label{fact:spanisunionofrects}The span of a rectangle $R$ is a
union of internally spanned rectangles.
\end{fact}

\begin{defn}
For $x\in\zz^{2}$, let $\overline{G}_{x}$ be the event that the
origin is in the span of $\left[-L,L\right]^{2}$ for $\eta$, but
not for $\eta^{x}$. $G_{x}$ is defined as the intersection of $\overline{G}_{x}$
with the event $\left\{ c_{x}=1\right\} $.
\end{defn}

First, note that legal flips of sites in the interior of a rectangle
or outside the rectangle cannot change its span. Therefore, $G_{x}=\emptyset$
for $x$ which is not on the inner boundary of $\left[-L,L\right]^{2}$.
\begin{claim}
Fix $x$ on the inner boundary of $\left[-L,L\right]^{2}$, and let
$\eta\in G_{x}$. Then $x$ and the origin belong to the same internally
spanned rectangle.
\end{claim}

\begin{proof}
Recalling \factref{spanisunionofrects}, we consider the internally
spanned rectangle containing the origin. If it didn't contain $x$,
the origin would be in the span of $\left[-L,L\right]^{2}$ also for
$\eta^{x}$, contradicting the definition of $G_{x}$.
\end{proof}
\begin{cor}
\label{cor:barrierof2f}Fix $x$ on the inner boundary of $\left[-L,L\right]^{2}$.
Then $\mu\left(G_{x}\right)\le\binom{L^{2}}{L/2}q^{\nicefrac{L}{2}}$.
\end{cor}

\begin{proof}
Assume without loss of generality that $x$ is on the right boundary.
Then there must be an internally spanned rectangle in $\left[-L,L\right]^{2}$
whose width is at least $L$. Since all sites of $\left[-L,L\right]^{2}$
are difficult, it cannot contain two consequent columns that are entirely
occupied, therefore it must contain at least $\frac{L}{2}$ empty
sites.
\end{proof}
We can use the same argument as in the proof of \claimref{taubisbig}.
Defining $G=\cup_{x}G_{x}$, this argument will tell us that the hitting
time $\tau_{G}$ is bigger than $C\,q^{-\nicefrac{L}{2}+1}$ with
probability that tends to $1$ as $q\rightarrow0$, for some constant
$C$ that depends on $L$. If we start with a configuration for which
the origin is not in the span of $\left[-L,L\right]^{2}$, it could
only be emptied after $\tau_{G}$ \textendash{} at the first instant
in which the span of $\left[-L,L\right]^{2}$ includes the origin,
$G_{x}$ must occur for the site that has just been flipped. Since
the probability to start with an entirely occupied $\left[-L,L\right]^{2}$
tends to $1$ as $q\rightarrow0$, \eqref{fa12d2kcmlowerbound} is
satisfied for $\underline{\alpha}=\frac{L}{2}-1$ .

In order to bound also the expected value of $\tau_{0}$ we will use
\propref{variationalprincipleforhittingtime}. Let us consider the
function
\[
f=\One_{0\text{ is not in the span of }\left[-L,L\right]^{2}}.
\]
We can bound its Dirichlet form using \corref{barrierof2f}:
\begin{align*}
\mathcal{D}f & =\mu\left(\sum_{x}c_{x}\text{Var}_{x}f\right)\le\mu\left(\sum_{x}c_{x}\,q\One_{G_{x}}\right)\\
 & \le q\,16L\,\binom{L^{2}}{L/2}\,q^{\nicefrac{L}{2}}=C_{L}\,q^{\nicefrac{L}{2}+1}.
\end{align*}
The expected value is bounded by the probability that all sites are
occupied \textendash 
\[
\mu f\ge\left(1-q\right)^{\left(2L+1\right)^{2}}.
\]

Now consider for some $\lambda\in\rr$ the rescaled function $\overline{f}=\lambda f$.
\begin{align*}
\mathcal{T}\overline{f} & =2\mu\overline{f}-\mathcal{D}\overline{f}\\
 & \ge2\lambda\left(1-q\right)^{\left(2L+1\right)^{2}}-\lambda^{2}\,C_{L}\,q^{\nicefrac{L}{2}+1}.
\end{align*}
The optimal choice of $\lambda$ is $\frac{\left(1-q\right)^{\left(2L+1\right)^{2}}}{C_{L}}q^{-\nicefrac{L}{2}-1}$,
which yields
\begin{align*}
\mathcal{T}\overline{f} & \ge\frac{\left(1-q\right)^{\left(2L+1\right)^{2}}}{C_{L}}q^{-\nicefrac{L}{2}-1}.
\end{align*}
\propref{variationalprincipleforhittingtime} and the fact that $\overline{f}$
vanishes on $\left\{ \eta_{0}=0\right\} $ implies that $\mu\left(\tau_{0}\right)\ge C_{L}^{\prime}q^{-\nicefrac{L}{2}-1}$,
and therefore $\ee_{\mu}\left(\tau_{0}\right)\ge q^{-\underline{\alpha}}$
for $q$ small enough.

Concerning \remref{exponentsaretruelyrandom}, note that for every
$\alpha$
\[
\nu\left(\underline{\alpha}\ge\alpha\right)\ge\nu\left(L\ge2\alpha+4\right)=\left(1-\pi\right)^{\left(4\alpha+9\right)^{2}}.
\]
In particular for $\alpha_{0}$ defined in the proof of the upper
bound $\nu\left(\underline{\alpha}\ge\alpha_{0}\right)>0$.

\subsection{Mixed threshold models on $\protect\zz^{d}$}

The argument above, for the case of $\zz^{2}$, works also in more
general settings, as long as the probability to be easy (i.e., threshold
$1$) is strictly positive.

The lower bound of the bootstrap percolation is immediate, since only
at scale $q^{-\nicefrac{1}{d}}$ it is possible to find an easy site.
The lower bound for the kinetically constrained model could be analyzed
similarly to the the two dimensional case using the methods of \cite{BaloghBollobasDuminilMorris1012sharpbpzd},
but in order to keep things simple we could take $\overline{\alpha}<1$,
which suffices since the rate at which the origin becomes empty is
always at most $q$.

For the upper bound of both the bootstrap percolation and the kinetically
constrained model we need to construct a path that will empty the
origin. First, note that we may assume that sites have either threshold
$1$ (easy) or threshold $d$ (difficult) with probabilities $\pi$
and $1-\pi$, for some $\pi>0$. In this case the path is described
explicitly in \cite{martinellitonitelli2016towardsuniversality}.
It is constructed for the FA$d$f model, but we will only need to
adapt the definitions there in order to take into account the easy
sites.

Fix $L$ that will be equal $n$ defined in the beginning of subsection
5.1 of \cite{martinellitonitelli2016towardsuniversality}, replacing
$q$ by $\pi$ and taking $\epsilon$ such that good boxes (see \defref{goodboxinzd}
that follows) percolate, and the origin belongs to the infinite cluster.
We then consider, just as before, an infinite path of good boxes starting
at the origin.
\begin{defn}
The easy bootstrap percolation will be the threshold $d$ bootstrap
percolation defined on $\zz^{d}$, with initial conditions in which
easy sites are empty and difficult sites are filled. A set $V\in\zz^{d}$
is \textit{easy internally spanned} if it is internally spanned for
this process.
\end{defn}

\begin{defn}
\label{def:goodboxinzd}A \textit{good} box $V=x+\left[L\right]^{d}$
is a box for which the event $G_{1}$ in Definition 5.4 of \cite{martinellitonitelli2016towardsuniversality}
occurs, replacing ``internally spanned'' by ``easy internally spanned''.
\end{defn}

\begin{defn}
As \textit{excellent} box is a box that, by adding a single easy site
at its corner, will be easy internally spanned. $p_{L}$ will be the
probability that the box $\left[L\right]^{d}$ is excellent. Note
that (as for the two dimensional case) $p_{L}$ is nonzero, and that
it does not depend on $q$.
\end{defn}

\begin{defn}
An \textit{essentially empty} box $V=x+\left[L\right]^{d}$ will correspond
to the event $G_{2}$ in Definition 5.4 of \cite{martinellitonitelli2016towardsuniversality}
\textendash{} it is a good box in which the first slice in any direction
is empty.
\end{defn}

Being good and being excellent are events measurable with respect
to the disorder $\omega$, whereas being essentially empty depends
also on the configuration of the empty and filled sites. The definition
of the bad event $B$ remains the same as in the previous section,
and taking $l=q^{-dL^{d-1}-1}$ its probability has the same decay.

With these definitions, replacing ``supergood'' by ``essentially
empty'', the proof in section 5 of \cite{martinellitonitelli2016towardsuniversality}
shows how to propagate an essentially empty box along the path of
good boxes, corresponding to figures \ref{fig:propagationgacolumn12}
and \ref{fig:rotatingcol12} in the two dimensional case. We may then
consider a path as the one of \lemref{fa12path}. That is, for any
configuration $\eta$ there is a path $\eta_{0},\dots,\eta_{N}$ with
flips $x_{0},\dots,x_{N-1}$ such that
\begin{enumerate}
\item $\eta_{0}=\eta$,
\item $\eta_{N}\in A$,
\item $\eta_{i+1}\in\eta_{i}^{x_{i}}$,
\item $c_{x_{i}}\left(\eta_{i}\right)=1$,
\item $N\le cL^{d}l$ for some $c>0$,
\item For all $i\le N$, $\eta_{i}$ differs from $\eta$ at at most $cL^{d-1}$
points, contained in at most two neighboring boxes.
\end{enumerate}
Applying the exact same argument as for the two dimensional case yields
the upper bounds.

\subsection{Mixed North-East and FA1f}

\subsubsection{Spectral gap}

This is the same argument as for the previous model \textendash{}
one can always find arbitrarily large regions of difficult sites,
so the gap is bounded by that of the north-east model. Since for the
parameters that we have chosen the north-east model is not ergodic,
it has $0$ gap \cite{cancrini2008kcmzoo}.

\subsubsection{Hitting time}

Let $A$ be the event $\left\{ \eta_{0}=0\right\} $. Recall \defref{taubar}
and let
\[
\tau=\overline{\tau}_{A}.
\]
The exponential tail of $\tau_{0}$ is a consequence of \propref{exponentialdecayofhittingtime},
so we are left with proving that $\nu\left(\tau\ge t\right)\le t^{\frac{c}{\log q}}$
for some constant $c$. We will do that by choosing a subgraph on
which we can estimate the gap, and then apply \claimref{gapofrestricteddynamics}.

Since $\pi$ is greater than the critical probability for the Bernoulli
site percolation, there will be an infinite cluster of easy sites
$\mathcal{C}$. We denote by $\mathcal{C}_{0}$ the cluster of the
origin surrounded by a path in $\mathcal{C}$. $\partial\mathcal{C}_{0}$
will be the outer boundary of $\mathcal{C}_{0}$, i.e., the sites
in $\mathcal{C}$ that have a neighbor in $\mathcal{C}_{0}$. Then,
we fix a self avoiding infinite path of easy sites $v_{0},v_{1},\dots$
starting with the sites of $\partial\mathcal{C}_{0}$. That is, $v_{0},\dots,v_{\left|\partial\mathcal{C}_{0}\right|}$
is a path that encircles $\mathcal{C}_{0}$, and then $v_{\left|\partial\mathcal{C}_{0}\right|+1},\dots$
continues to infinity. We will denote $\mathcal{V}=\left\{ v_{i}\right\} _{i\in\nn}$.
Let $H=\mathcal{V}\cup\mathcal{C}_{0}$, and consider the restricted
dynamics $\mathcal{L}_{H}$ introduced in \defref{restricteddynamics}.
We split the dynamics in two \textendash{} for some local function
$f$ on $H$
\begin{align*}
\mathcal{L}_{H}f & =\mathcal{L}^{\mathcal{C}_{0}}f+\mathcal{L}^{\mathcal{V}}f,\\
\mathcal{L}^{\mathcal{V}} & =\sum_{i\in\nn}c_{v_{i}}^{H}\left(\mu_{v_{i}}f-f\right),\\
\mathcal{L}^{\mathcal{C}_{0}} & =\sum_{x\in\mathcal{C}_{0}}c_{x}^{H}\left(\mu_{x}f-f\right).
\end{align*}

Note that the boundary conditions of the $\mathcal{C}_{0}$ dynamics
depend on the state of the vertices in $\mathcal{V}$ and vice versa.
We will denote by $\text{\ensuremath{\mathcal{L}}}_{0}^{\mathcal{C}_{0}}$
the $\mathcal{C}_{0}$ dynamics with empty boundary conditions and
by$\text{\ensuremath{\mathcal{L}}}_{1}^{\mathcal{V}}$ the $\mathcal{V}$
dynamics with occupied boundary conditions. All generators come with
their Dirichlet forms carrying the same superscript and subscript.

We will bound the gap of $\mathcal{L}_{H}$ using the gaps of $\mathcal{L}_{1}^{\mathcal{V}}$,
$\mathcal{L}_{0}^{\mathcal{C}_{0}}$ and the following block dynamics:
\[
\mathcal{L}^{b}f=\left(\mu_{\mathcal{V}}\left(f\right)-f\right)+\One_{\partial\mathcal{C}_{0}\text{ is empty}}\left(\mu_{\mathcal{C}}f-f\right).
\]
Denote the spectral gaps of $\mathcal{L}_{1}^{\mathcal{V}}$, $\mathcal{L}_{0}^{\mathcal{C}_{0}}$,$\mathcal{L}^{b},\mathcal{L}_{H}$
by $\gamma_{1}^{\mathcal{V}}$,$\gamma_{0}^{\mathcal{C}_{0}}$,$\gamma^{b}$,$\gamma_{H}$.

By Proposition 4.4 of \cite{cancrini2008kcmzoo}:
\begin{claim}

\[
\gamma^{b}=1-\sqrt{1-q^{\left|\partial\mathcal{C}_{0}\right|}},
\]
i.e., $\text{Var}f\le\frac{1}{1-\sqrt{1-q^{\left|\partial\mathcal{C}_{0}\right|}}}\mathcal{D}^{b}f$
for any local function $f$.
\end{claim}

Let us now use this gap in order to relate $\gamma_{H}$ to $\gamma^{\mathcal{V}}$
and $\gamma^{\mathcal{C}_{0}}$:
\begin{claim}

\[
\gamma_{H}\ge\gamma^{b}\min\left\{ \gamma_{1}^{\mathcal{V}},\gamma_{0}^{\mathcal{C}_{0}}\right\} .
\]
\end{claim}

\begin{proof}
Fix a non-constant local function $f$.
\begin{align*}
\text{Var}f & \le\frac{1}{\gamma^{b}}\mathcal{D}^{b}f=\frac{1}{\gamma^{b}}\left[\mu\left(\text{Var}_{\mathcal{V}}f\right)+\mu\left(\One_{\partial\mathcal{C}_{0}\text{ is empty}}\text{Var}_{\mathcal{C}}f\right)\right]\\
 & \le\frac{1}{\gamma^{b}}\left[\frac{1}{\gamma_{1}^{\mathcal{V}}}\mu\left(\mathcal{D}_{1}^{\mathcal{V}}f\right)+\frac{1}{\gamma_{0}^{\mathcal{C}_{0}}}\mu\left(\One_{\partial\mathcal{C}_{0}\text{ is empty}}\mathcal{D}^{\mathcal{C}_{0}}f\right)\right]\\
 & \le\frac{1}{\gamma^{b}}\max\left\{ \frac{1}{\gamma_{1}^{\mathcal{V}}},\frac{1}{\gamma_{0}^{\mathcal{C}_{0}}}\right\} \mathcal{D}_{H}f.
\end{align*}
\end{proof}
We are left with estimating $\gamma_{1}^{\mathcal{V}}$ and $\gamma_{0}^{\mathcal{C}_{0}}$.
\begin{claim}
There exists $C>0$ such that $\gamma_{1}^{\mathcal{V}}\ge Cq^{3}$.
\end{claim}

\begin{proof}
The Dirichlet form $\mathcal{D}_{1}^{\mathcal{V}}$ is dominated by
the Dirichlet form of FA1f on $\zz_{+}$, and that dynamics has spectral
gap which is proportional to $q^{3}$ (see \cite{cancrini2008kcmzoo}).
\end{proof}
For $\gamma_{0}^{\mathcal{C}_{0}}$ we will use the bisection method,
comparing the gap on a box to that of a smaller box. For $L\in\nn$,
let $\mathcal{L}_{L}^{\text{NE}}$ be the generator of the north-east
dynamics in the box $\left[L\right]^{2}$ with empty boundary (for
the north east model this is equivalent to putting empty boundary
only above and to the right). Denote its gap by $\gamma_{\left[L\right]^{2}}^{\text{NE}}$.
By monotonicity we can restrict the discussion to this dynamics, i.e.,
\begin{equation}
\gamma_{0}^{\mathcal{C}_{0}}\ge\gamma_{\text{diam }\mathcal{C}_{0}}^{\text{NE}}.\label{eq:expandingnetosquare}
\end{equation}

We will now bound $\gamma^{\text{NE}}$ (see also Theorem 6.16 of
\cite{cancrini2008kcmzoo}).
\begin{claim}
$\gamma_{\left[L\right]^{2}}^{\text{NE}}\ge e^{3\log q\,L}$.
\end{claim}

\begin{proof}
We will prove the result for $L_{k}=2^{k}$ by induction on $k$.
Then monotonicity will complete the argument for all $L$. Consider
the box $\left[L_{k}\right]^{2}$, and divide it in two rectangles
\textendash{} $R_{-}=\left[L_{k-1}\right]\times\left[L_{k}\right]$
and $R_{+}=\left[L_{k-1}+1,L_{k}\right]\times\left[L_{k}\right]$.
We will run the following block dynamics
\begin{align*}
\mathcal{L}^{b\text{NE}}f & =\left(\mu_{R_{+}}f-f\right)+\One_{\partial_{-}R_{+}\text{ is empty}}\left(\mu_{R_{-}}f-f\right),
\end{align*}
where $\partial_{-}R_{+}$ is the inner left boundary of $R_{+}$.
Again, by Proposition 4.4 of \cite{cancrini2008kcmzoo},
\begin{align*}
\text{gap}\left(\mathcal{L}^{b\text{NE}}\right) & =1-\sqrt{1-\mu\left(\One_{\partial_{-}R_{+}\text{ is empty}}\right)}\\
 & =1-\sqrt{1-q^{L_{k}}}.
\end{align*}
Therefore for every local function $f$
\begin{align*}
\text{Var}f & \le\frac{1}{1-\sqrt{1-q^{L_{k}}}}\mathcal{D}^{b\text{NE}}f\\
 & =\frac{1}{1-\sqrt{1-q^{L_{k}}}}\mu\left(\text{Var}_{R_{+}}f+\One_{\partial_{-}R_{+}\text{ is empty}}\text{Var}_{R_{-}}f\right)\\
 & \le\frac{1}{1-\sqrt{1-q^{L_{k}}}}\mu\left(\frac{1}{\gamma_{R_{+}}^{\text{NE}}}\mathcal{D}_{R_{+}}^{\text{NE}}f+\frac{1}{\gamma_{R_{-}}^{\text{NE}}}\mathcal{D}_{R_{-}}^{\text{NE}}f\right),
\end{align*}
where $\gamma_{R}^{\text{NE}},\mathcal{D}_{R}^{\text{NE}}$ are the
spectral gap and Dirichlet form of the north-east dynamics in $R$
with empty boundary conditions for any fixed rectangle $R$. We see
that
\[
\gamma_{\left[L_{k}\right]^{2}}^{\text{NE}}\ge\left(1-\sqrt{1-q^{L_{k}}}\right)\gamma_{\left[L_{k-1}\right]\times\left[L_{k}\right]}^{\text{NE}}.
\]
If we repeat the same argument dividing $\left[L_{k-1}\right]\times\left[L_{k}\right]$
into the rectangles $\left[L_{k-1}\right]\times\left[L_{k-1}\right]$
and $\left[L_{k-1}\right]\times\left[L_{k-1}+1,L_{k}\right]$, we
obtain

\[
\gamma_{L_{k-1}\times L_{k}}^{\text{NE}}\ge\left(1-\sqrt{1-q^{L_{k-1}}}\right)\gamma_{\left[L_{k-1}\right]^{2}}^{\text{NE}}.
\]
Hence,
\[
\log\gamma_{\left[L_{k}\right]^{2}}^{\text{NE}}\ge\log\gamma_{\left[L_{k-1}\right]^{2}}^{\text{NE}}+2^{k}\log q-\log4,
\]
yielding
\[
\log\gamma_{\left[L_{k}\right]^{2}}^{\text{NE}}\ge\log q\sum_{n=1}^{k}2^{n}-k\log4
\]
which finishes the proof.
\end{proof}
We can now put everything together. Let $L$ be the diameter of $\mathcal{C}_{0}$.
By the second part of \claimref{gapofrestricteddynamics}
\begin{align}
\tau & \le\frac{1+q}{q}\,\frac{1}{\gamma^{r}}\le\frac{1+q}{q}\,\frac{1}{\left(1-\sqrt{1-q^{\left|\mathcal{C}\right|}}\right)\min\left\{ \gamma_{1}^{\mathcal{V}},\gamma_{0}^{\mathcal{C}_{0}}\right\} }\label{eq:boundingtauforneinfixedomega}\\
 & \le q^{-4L-1}.\nonumber 
\end{align}

Finally, we will use the sharpness of the phase transition for the
site percolation on the dual graph (see \cite{AizenmanBarsky87percosharpness,DuminilCopinTassion16percoshaprness}):
\begin{claim}
There exists a positive constant $c_{2}$ that depends on $\pi$ such
that $\nu\left(L\ge D\right)\le e^{-c_{2}D}$ for any $D\in\nn$. 
\end{claim}

Using this claim and \eqref{boundingtauforneinfixedomega}

\begin{align*}
\nu\left(\tau\ge t\right) & \le\nu\left(q^{-4L-1}\ge t\right)=\nu\left(L\ge\frac{\log t}{4\log\frac{1}{q}}-\frac{1}{4}\right)\\
 & \le\,C\,t^{\nicefrac{c}{\log q}}.
\end{align*}

\section{Conclusions and further questions}

We have seen here two simple examples of KCMs in random environments.
These examples show that when the environment has some rare remote
``bad'' regions the relaxation time fails to describe the true observed
time scales of the system. Since the dynamics is not attractive, also
techniques such as monotone coupling and censoring cannot be applied
to these models. In order to overcome this difficulty we considered
the hitting time $\tau_{0}$, which on one hand describes a physically
measurable observable, and on the other hand could be studied using
variational principles. We formulated some tools based on these variational
principles and used them in order to understand the behavior of $\tau_{0}$
in both models.

For future research, one may try to apply these tools on kinetically
constrained models in more types of random environment, such as the
polluted lattice, more general mixture of constraints, and models
on random graphs.

There are also some questions left open considering the models studied
here. For the first model, it is natural to conjecture that $\tau_{0}$
scales as $q^{-\alpha}$ for some random $\alpha$. We can also look
at the $\pi$ dependence of $\alpha$ \textendash{} we know that when
there are not many easy sites this time should become larger, until
it reaches the FA2f time when $\pi=0$. Looking at the proof of \thmref{scalingoftimeforfa12},
we can see that $\overline{\alpha}$ scales like $\frac{\log\nicefrac{1}{\pi}}{\pi}$,
and $\underline{\alpha}$ scales like $\pi^{-\nicefrac{1}{2}}$. It
seems more likely, however, that the actual exponent $\alpha$ behaves
like $\frac{1}{\pi}$ \textendash{} the $\log\frac{1}{\pi}$ of the
lower bound comes up also in the proof bounding the gap of the FA2f
model \cite{martinellitonitelli2016towardsuniversality}, and also
there it is conjectured that it does not appear in the true scaling.
In fact, if we use the path that we have chosen in order to bound
$\tau_{0}$ also for the bootstrap percolation, we will have the same
$\log\frac{1}{q}$ factor, and in this case it is known that it does
not appear in the correct scaling. The lower bound of $\pi^{-\nicefrac{1}{2}}$
could be improved, with the price of complicating the proof. In the
proof we assume that $\left[-L,L\right]^{2}$ contains only difficult
sites. If we require instead that enough lines and columns are entirely
difficult (but not necessarily the entire square), we can obtain a
bound that scales as $\frac{1}{\pi}$. It is worth noting here that
for the bootstrap percolation, by repeating the arguments of \cite{aizenmanlebowitz1988metastabilityz2,holroyd2003sharp}
with some minor adaptations, we can show that the scaling of the prefactor
$a$ with $\pi$ is between $e^{c/\pi}$ and $e^{C/\pi}$ for $c,C>0$.
\begin{conjecture}
For the mixed threshold FA model, $\nu$-almost surely the limit $\lim_{q\rightarrow0}\frac{\log\tau_{0}}{\log\nicefrac{1}{q}}$
exists. Its value $\alpha$ is a random variable whose law depends
on $\pi$. Moreover, the law of $\pi\alpha$ converges (in some sense)
to a non-trivial law as $\pi$ tends to $0$.
\end{conjecture}

The mixed North-East and FA1f model also raises many questions, among
them finding the critical probability $q_{c}$, and characterizing
the behavior of both the bootstrap percolation and the KCM in the
different parameter regimes.

\section*{Acknowledgments}

I would like to thank Cristina Toninelli and Fabio Martinelli for
the helpful discussions and useful advice. I acknowledge the support
of the ERC Starting Grant 680275 MALIG.

\bibliographystyle{plain}
\bibliography{random_constraint}

\end{document}